\documentclass{amsart}
\usepackage{amsfonts,amssymb,stmaryrd,amscd,amsmath,latexsym,amsbsy}
\usepackage{hyperref}
\numberwithin{equation}{section}
\theoremstyle{plain}
\newtheorem{thm}{Theorem}[section]

\newtheorem{prop}[thm]{Proposition}

\newtheorem{defi}[thm]{Definition}
\newtheorem{lem}[thm]{Lemma}
\newtheorem{cor}[thm]{Corollary}

\theoremstyle{remark}
\newtheorem{rema}[thm]{Remark}

\title[Trigonometric Cherednik algebra at critical level]{Trigonometric 
Cherednik algebra at critical level and quantum many-body problems}
\author{E. Emsiz, E. M. Opdam and J. V. Stokman}
\address{E. Emsiz, Instituto de Matem\'atica y F\'isica, 
Universidad de Talca, Casilla 747, Talca, Chile}
\email{eemsiz@inst-mat.utalca.cl}
\address{E. M. Opdam and J. V. Stokman, 
KdV Institute for Mathematics, Universiteit van Amsterdam,
Plantage Muidergracht 24, 1018 TV Amsterdam, The Netherlands.}
\email{opdam@science.uva.nl, jstokman@science.uva.nl}
\begin{document}

\begin{abstract}
For any module over the affine Weyl group
we construct a representation of the associated trigonometric Cherednik
algebra $A(k)$ at critical level in terms of Dunkl type operators.  Under
this representation the center of $A(k)$ produces
quantum conserved integrals for root system generalizations of
quantum spin-particle systems on the circle with delta function
interactions.
This enables us to translate the spectral problem of such a quantum
spin-particle
system to questions in the representation theory of $A(k)$. We use this
approach
to derive the associated Bethe ansatz equations. They are expressed in
terms of the normalized intertwiners of $A(k)$.

\end{abstract}

\maketitle
\section{Introduction}

The trigonometric Cherednik algebra \cite{C0,C} 
depends on a root system $R$, a multiplicity function $k$, 
and a level $c$. For noncritical level $c\not=0$
it is an indispensable tool in the analysis 
of quantum Calogero-Moser systems with trigonometric potentials.  
In this paper we use the trigonometric Cherednik algebra
$A(k)$ at critical level $c=0$ to analyze the 
root system generalizations of quantum spin-particle 
systems on the circle with delta function interactions. 

The study of one dimensional quantum spin-particle systems with delta
function interactions goes back to Lieb and Liniger \cite{LL}.
These systems are particularly well studied and 
have an amazingly rich structure, see, e.g., 
\cite{LL,YY,McG,Y,S,Ga}, to name just a few.
They have been successfully analyzed by Bethe ansatz methods and
by quantum inverse scattering methods. We want to advertise here
yet another technique which is based on degenerate Hecke algebras.
It allows us to extend the Bethe ansatz techniques to 
quantum Hamiltonians with delta function potentials 
along the root hyperplanes of any (affine) root system. This 
builds on many earlier works, see, e.g., \cite{Ga,GS,G,P,CY,MW,HK,HO,EOS}.
In special cases the associated quantum system describes 
one dimensional quantum spin-particles with pair-wise delta function
interactions and with boundary reflection terms.
   
To clarify the interrelations between these techniques
we feel that it is instructive 
to start with a short discussion about the analysis of
the quantum spin-particle systems on the circle with delta
function interactions using the classical Bethe ansatz method. 
It goes back to the work of Lieb and Liniger \cite{LL} in case of the
quantum Bose gas. The extension of these techniques to quantum particles 
with spin was considered, amongst others, by McGuire \cite{McG0,McG}, 
Flicker and Lieb \cite{FL}, Gaudin \cite{Ga0} and Yang \cite{Y}.

Consider the natural action of $S_n\ltimes\mathbb{Z}^n$ on
$\mathbb{R}^n$ by permutations and translations and choose a representation
$\rho: S_n\ltimes\mathbb{Z}^n\rightarrow \textup{GL}_{\mathbb{C}}(M)$
(it encodes the spin of the quantum particles). 
Consider the quantum Hamiltonian 
\[
\mathcal{H}_k^M=-\Delta-k\sum_{\stackrel{1\leq i<j\leq n}{m\in\mathbb{Z}}}
\delta(x_i-x_j+m)\rho(s_{i,j;m})
\]
with $\Delta$ the Laplacean on $\mathbb{R}$, $k\in\mathbb{C}$
a coupling constant, $\delta$ Dirac's delta function 
and $s_{i,j;m}\in S_n\ltimes\mathbb{Z}^n$
the orthogonal reflection in the affine hyperplane
\[
\{x\in\mathbb{R}^n \, | \, x_j-x_i=m\}.
\] 
With the present choice of notations, coupling constant
$k<0$ (respectively $k>0$) corresponds to repulsive (respectively
attractive) delta function interactions between the quantum particles.

A fundamental domain for the action of $S_n\ltimes\mathbb{Z}^n$ on 
$\mathbb{R}^n$ is 
\[
C_+=\{x\in\mathbb{R}^n \, | \, x_1>x_2>\cdots>x_n>x_1-1\}.
\] 
The Bethe hypothesis in this context (see \cite{LL}) 
is to look for eigenfunctions
of $\mathcal{H}_k^M$ that have an expansion in plane waves 
in each $S_n\ltimes\mathbb{Z}^n$-translate of $C_+$.
The hypothesis is justified by the following result.
\begin{thm}\label{introTHM}
Fix $\lambda\in\mathbb{C}^n$ such that $\lambda_i\not=\lambda_j$
for $i\not=j$. Choose furthermore $m_w\in M$ ($w\in S_n$).
There exists a unique $M$-valued continuous function
$f_\lambda$ on $\mathbb{R}^n$ satisfying
\begin{enumerate}
\item[{\bf (i)}] $f_\lambda$ is an eigenfunction 
of $\mathcal{H}_k^M$ (in the weak sense)
with eigenvalue given by $-\lambda^2:=-\sum_{j=1}^n\lambda_j^2$.
\item[{\bf (ii)}] On each $S_n\ltimes\mathbb{Z}^n$-translate of $C_+$,
$f_\lambda$ can be expressed as a sum of plane waves $e^{w\lambda}$
($w\in S_n$) with coefficients in $M$.
\item[{\bf (iii)}] $f_\lambda|_{C_+}=\sum_{w\in S_n}m_we^{w\lambda}|_{C_+}$.
\end{enumerate}
\end{thm}
The proof of the theorem encompasses an explicit recipe how to propagate 
$\sum_{w\in S_n}m_we^{w\lambda}|_{C_+}$ 
to the eigenfunction $f_\lambda$
of $\mathcal{H}_k^M$. It is based on the reformulation of the 
spectral problem $\mathcal{H}_k^Mf_\lambda=-\lambda^2f_\lambda$ as a 
boundary value problem. 
The uniqueness of $f_\lambda$ is the subtle point, we discuss it
in Section \ref{Spectralsection}. 

Fix $\lambda\in\mathbb{C}^n$ generic (to be made precise
in the main text) and consider 
the diagonal $S_n\ltimes\mathbb{Z}^n$-action $(w\cdot f)(x):=
\rho(w)f(w^{-1}x)$ on the space of $M$-valued functions on $\mathbb{R}^n$.
The study of the $S_n$-invariant eigenfunctions $f_\lambda$ has led
to the discovery of the famous Yang-Baxter equation with spectral parameters
as follows \cite{McG0,Y}.
Denote $s_{i,j}=s_{i,j;0}\in S_n$ and consider the elements 
\[
Y_{i,j}^k(u)=\frac{u+ks_{i,j}}{u-k},\qquad 1\leq i<j\leq n
\]
in the group algebra $\mathbb{C}[S_n]$ of $S_n$ 
for generic $u\in\mathbb{C}$. The $S_n$-invariance of the
eigenfunction $f_\lambda$ is equivalent to the plane wave coefficients
$m_w$ being of the form $m_w=\rho(J_w^k(\lambda))m$ ($m\in M$)
with $J_w^k(\lambda)$ the unique 
elements in $\mathbb{C}[S_n]$ satisfying $J_1^k(\lambda)=1$ and 
\begin{equation}\label{rr}
J_{s_{i,i+1}w}^k(\lambda)=Y_{i,i+1}^k(\lambda_{w^{-1}(i)}-\lambda_{w^{-1}(i+1)})
s_{i,i+1}J_w^k(\lambda)
\end{equation}
for $1\leq i<n$ and $w\in S_n$. We denote the corresponding $S_n$-invariant
eigenfunction by $f_\lambda^m$. The consistency of \eqref{rr} is equivalent to 
the $Y_{i,j}$ being unitary solutions of the Yang-Baxter equation,
\begin{equation*}
\begin{split}
Y_{i,j}^k(u)Y_{i,l}^k(u+v)Y_{j,l}^k(v)&=
Y_{j,l}^k(v)Y_{i,l}^k(u+v)Y_{i,j}^k(u),\qquad 1\leq i<j<l\leq n,\\
Y_{i,j}^k(-u)&=Y_{i,j}^k(u)^{-1},\qquad\qquad\qquad\qquad\qquad
1\leq i<j\leq n.
\end{split}
\end{equation*}

We now translate the requirement that 
$f_\lambda^m$ is $\mathbb{Z}^n$-invariant to explicit
conditions on $m\in M$. 
The group $S_n\ltimes\mathbb{Z}^n$ is generated by $S_n$
and one additional element $\pi$, which we characterize here by its
action on $\mathbb{R}^n$,
\[\pi(x_1,\ldots,x_{n-1},x_n)=(x_n+1,x_1,\ldots,x_{n-1}).
\]
There exists a unique extension of the $J_w^k(\lambda)$  
to elements $J_w^k(\lambda)$ ($w\in S_n\ltimes\mathbb{Z}^n$)
in the group algebra $\mathbb{C}[S_n\ltimes\mathbb{Z}^n]$ 
of $S_n\ltimes\mathbb{Z}^n$ satisfying $J_\pi^k(\lambda)=\pi$ and satisfying
the cocycle relation
\[J_{vw}^k(\lambda)=J_v^k((Dw)\lambda)J_w(\lambda)\qquad
\forall\, v,w\in S_n\ltimes\mathbb{Z}^n,
\]
where $D: S_n\ltimes\mathbb{Z}^n\rightarrow S_n$ is the
natural surjective group homomorphism (omitting the translation part).
Then $f_\lambda^m$ is $\mathbb{Z}^n$-invariant if and only if 
$m\in M$ satisfies the Bethe ansatz equations
\[\rho(J_y^k(\lambda))m=e^{\lambda(y)}m,\qquad\forall\, y\in\mathbb{Z}^n,
\]
see Theorem \ref{BAEthm}.
The Bethe ansatz equations \eqref{BAE} for special modules $M$ were
derived in, e.g., \cite{LL,McG,Y,Ga0,S}.

Our results will show that for root system $R$ of type $A$,
the trigonometric Cherednik
algebra $A(k)$ at critical level $c=0$ enters the analysis of
these quantum systems in three closely related ways: 
\begin{enumerate}
\item[{\bf (i)}] The quantum Hamiltonian $\mathcal{H}_k^M$
can be constructed from a family of commuting
Dunkl type differential-reflection operators. These Dunkl operators, together
with the diagonal $S_n\ltimes\mathbb{Z}^n$-action, give a 
presentation of $A(k)$.
\item[{\bf (ii)}] The propagation procedure, which extends
a plane wave $\sum_{w\in S_n}m_we^{w\lambda}|_{C_+}$ to the 
eigenfunction $f_\lambda$ of $\mathcal{H}_k^M$, 
is governed by a representation of
$S_n\ltimes\mathbb{Z}^n$ defined in terms of integral-reflection operators.
Together with the constant coefficient differential operators it gives
another presentation of $A(k)$.
\item[{\bf (iii)}] The cocycle $\{J_w^k(\lambda)\}_{w\in S_n\ltimes\mathbb{Z}^n}$ 
comes from the action of the normalized intertwiners of $A(k)$
on principal series modules of $A(k)$.
\end{enumerate}
In this paper we discuss {\bf (i--iii)} for arbitrary root
systems $R$. It is subsequently used to apply the Bethe ansatz methods
to the root system generalizations of the spin-particle systems on
the circle with delta function interactions. It leads in particular 
to the Bethe ansatz equations for any root system $R$ 
(see Theorem \ref{BAEthm}).  

\section{Notations}


\subsection{Orthogonal reflections in affine hyperplanes}

Let $V$ be a
Euclidean vector space with scalar
product $\langle\cdot,\cdot\rangle$.
The linear dual $V^*$ inherits from $V$ the structure of a
Euclidean vector space and we also write
$\langle\cdot,\cdot\rangle$ for the associated
scalar product on $V^*$. In this and the next subsection $B$ denotes
a commutative unital noetherian $\mathbb{R}$-algebra.
For any real vector space $M$ we use $M_B$ as a shorthand notation
for the $B$-module $B\otimes_\mathbb{R}M$.
We extend the scalar products on $V$ and on $V^*$ to $B$-bilinear forms on
$V_B$ and $V_{B}^*$, which we still denote by $\langle \cdot,\cdot\rangle$.
Let $P(V)$ denote the algebra of real polynomial functions on $V$.
We regularly identify $P_B(V):=P(V)_B$ with the symmetric algebra
$S(V_B^*)$ of $V_B^*$ by interpreting $\xi\in V_B^*$ as the
$B$-valued polynomial $v\mapsto \xi(v)$ on $V$.

Consider the space
$\hbox{Aff}_B(V):=\hbox{Aff}(V)_B$ of $B$-valued affine linear functions
on $V$. It is the subspace of $P_B(V)$ consisting of $B$-valued polynomials
of degree $\leq 1$ on $V$. Under the natural identification
$P_B(V)\simeq S(V_B^*)$, the affine linear function $\phi\in\textup{Aff}_B(V)$
identifies with $\xi+\lambda 1\in V_B^*\oplus B1$ ($\xi\in V_B^*$,
$\lambda\in B$) where $\xi$ is the gradient $D\phi$ of $\phi$ and
$\lambda=\phi(0)$, i.e.
\[\phi(v)=\xi(v)+\lambda,\qquad \forall\, v\in V.
\]

The co-vector $\xi^\vee\in V$
associated to $\xi\in V^*\setminus \{0\}$ is
defined by
\[\eta(\xi^\vee)=2\frac{\langle \eta,\xi\rangle}{\langle \xi,\xi\rangle},\qquad
\forall\, \eta\in V^*.
\]
For $\phi\in \hbox{Aff}(V)$ with nonzero gradient,
the map $s_\phi: V\rightarrow V$ defined by
\[s_\phi(v)=v-\phi(v)(D\phi)^\vee,
\qquad v\in V
\]
is the orthogonal reflection in the affine hyperplane
$V_\phi=\{v\in V \, | \, \phi(v)=0\}$.
Observe that $s_\phi$ is linear if $\phi$ is linear,
in which case we also write $s_\phi$ for its $B$-linear extension
$s_\phi: V_B\rightarrow V_B$.

For $v\in V_B$ we define translation operators
$t_v: V_B\rightarrow V_B$ by
\[
t_v(v^\prime)=v+v^\prime, \qquad v^\prime\in V_B.
\]
Note that $f\circ t_v=t_{f(v)}\circ f$ for $B$-linear
mappings $f: V_B\rightarrow V_B$.
Furthermore, for
$\phi\in\textup{Aff}(V)$ with nonzero gradient we have
\begin{equation}\label{decompWtranslation}
s_\phi=s_{D\phi}t_{\phi(0)(D\phi)^\vee}.
\end{equation}

\subsection{Root systems}\label{Rootsection}

Let $R\subset V^*$ be a reduced, crystallographic root system in $V^*$.
We assume that $R$ is irreducible when considered as root system
in $\textup{span}_{\mathbb{R}}\{R\}\subset V^*$.
The associated co-root system is
$R^\vee=\{\alpha^\vee\}_{\alpha\in R}\subset V$.

Note that we do not require that $R$ spans $V^*$. This allows us
for instance to consider the root system
$R=\{\epsilon_i-\epsilon_j\}_{1\leq i\not=j\leq n}$
of type $A_{n-1}$ with ambient Euclidean space $V^*=\mathbb{R}^n$,
where $\{\epsilon_i\}_{i=1}^n$ is the standard orthonormal basis
of $\mathbb{R}^n$.

The Weyl group $W$ is the subgroup of
$O(V)$ generated by the orthogonal reflections
$s_\alpha$ in the root hyperplanes $V_\alpha$ ($\alpha\in R$),
hence $V$ has a canonical $\mathbb{R}[W]$-module structure.
By extension of scalars we consider $V_B$ as a $B[W]$-module.
In the dual module $V_B^*$ the action of $s_\alpha$ is given explicitly
by
\[s_\alpha(\xi)=\xi-\xi(\alpha^\vee)\alpha,\qquad \xi\in V_B^*
\]
for $\alpha\in R$.

We fix a full lattice $X\subset V^*$ satisfying the following two properties:
\begin{enumerate}
\item[{\it (i)}] $X$ contains the root lattice $Q$ of $R$.
\item[{\it (ii)}] The lattice $Y\subset V$ dual to $X$
contains the co-root lattice $Q^\vee$ of $R^\vee$.
\end{enumerate}
The conditions {\it (i)} and {\it (ii)} imply that
$X$ and $Y$ are $W$-invariant. A key example
is $R$ the root system of type $A_{n-1}$ with $V^*=\mathbb{R}^n$
and $X=\bigoplus_{j=1}^n\mathbb{Z}\epsilon_j$.

\begin{defi}
The extended affine Weyl group associated to the above data
is $W^a=W\ltimes Y$. The affine Weyl group is the normal subgroup
$W\ltimes Q^\vee$ of $W^a$.
\end{defi}
We will identify $W^a$ with the subgroup $\{wt_y\}_{w\in W, y\in Y}$
of the group of isometries of $V$.
The gradient map $D: W^a\rightarrow W$ is the surjective group
homomorphism defined by $D(wt_y)=w$ ($w\in W$, $y\in Y$).

Fix $b\in B$. For $w\in W$ and $y\in Y$ we set
\[
(wt_y)^{(b)}(v):=(wt_{by})(v),\qquad v\in V_B.
\]
This defines a left $W^a$-action on $V_B$.
Observe that $w^{(1)}=w$ and
$w^{(0)}=Dw$ for $w\in W^a$.
Furthermore,
\[w^{(b)}\circ \lambda_b=\lambda_b\circ w,\qquad w\in W^a,
\]
where $\lambda_b(v):=bv$ ($v\in V_B$).

By transposition the $W^a$-action $w\mapsto w^{(b)}$ on $V_B$
gives rise to left actions of $W^a$ on
$P_B(V)\simeq S(V_B^*)$ by $B$-algebra automorphisms.
The subspace $\hbox{Aff}_B(V)$ of $P_B(V)\simeq S(V_B^*)$ is
$W^a$-invariant for all scaling factors $b\in B$ since
\begin{equation}\label{help}
(wt_{by})(\xi+\lambda 1)=w(\xi)+(\lambda-b\xi(y))1
\end{equation}
for $w\in W$, $y\in Y$, $\xi\in V_B^*$ and $\lambda\in B$.

\begin{defi}
The subset $R^a=R+\mathbb{Z}1$ of $V^*+\mathbb{R}1$
is the affine root system associated to $R$.
\end{defi}
By \eqref{help},
\[(wt_y)(a)=w(\alpha)+(m-\alpha(y))1,\qquad a=\alpha+m1\in R^a
\]
for $w\in W$ and $y\in Y$. By assumption the roots $\alpha\in R$
are contained in the lattice $X$,
hence $\alpha(y)\in\mathbb{Z}$ for $\alpha\in R$ and $y\in Y$.
Thus $R^a$ is $W^a$-invariant.

By \eqref{decompWtranslation} we have
\[s_a=s_\alpha t_{m\alpha^\vee},\qquad a=\alpha+m1\in R^a.
\]
Consequently the affine Weyl group $W\ltimes Q^\vee$
is generated by the orthogonal reflections $s_a$
in the affine hyperplanes $V_a$,
where $a$ runs over the set $R^a$ of affine roots. Note furthermore
that $ws_aw^{-1}=s_{wa}$ and $Ds_a=s_{Da}$ for all $a\in R^a$ and
$w\in W^a$.

We fix a basis $F$ of $R$. Denote $R^{\pm}$ for the associated
positive and negative roots, and $\theta\in R$ for the highest root
of $R$ with respect to $F$.
The basis $F$ of $R$ extends to a basis $F^a$ of $R^a$ by
adding the simple affine root $a_0:=-\theta+1$ to $F$.
The associated simple reflection $s_{a_0}$ will be denoted by $s_0$.
The positive and negative affine roots are
$R^{a,+}=(R+\mathbb{Z}_{>0})\cup R^{+}$ and $R^{a,-}=-R^{a,+}$
respectively. Define the length of $w\in W^a$ by
\[l(w)=\#\bigl(R^{a,+}\cap w^{-1}(R^{a,-})\bigr).
\]
The following proposition is well known.

\begin{prop}
Set $\Omega=\{w\in W^a \, | \, l(w)=0\}$.\\
{\bf (i)} $\Omega$ is an abelian subgroup of $W^a$, isomorphic to
$W^a/(W\ltimes Q^\vee)\simeq Y/Q^\vee$.\\
{\bf (ii)} $\omega\in\Omega$ permutes
the set $F^a$ of simple roots.
\end{prop}

\section{The trigonometric Cherednik algebra}


\subsection{The algebra $H^a_L$}

Let $L=\mathbb{C}[\mathbf{c},\mathbf{k}]$
be the polynomial algebra in the indeterminates
$\mathbf{c}$ and the $\mathbf{k}_{\mathcal{O}}$,
where $\mathcal{O}$ runs over the set of $W^a$-orbits in $R^a$.
We write $\mathbf{k}_b=\mathbf{k}_{W^a(b)}$ for $b\in R^a$,
and $\mathbf{k}_0=\mathbf{k}_{a_0}$.

Depending on the root system $R$ and
the lattice $Y$, we thus have one, two or three commuting indeterminates
$\mathbf{k}_{\mathcal{O}}$. Concretely, if $R$ is not of
type $C_n$ ($n\geq 1$) then the $W^a$-orbits are
of the form $W^a(\beta)$ with $\beta$ running through a complete
set of representatives
of the $W$-orbits of $R$.
The same is true for $R$ of type $C_n$ if $\theta(Y)=\mathbb{Z}$.
For $R$ of type $C_n$ ($n\geq 1$) and $\theta(Y)=2\mathbb{Z}$ we have
two (if $n=1$) or three (if $n\geq 2$)
$W^a$-orbits in $R^a$, namely $W^a(a_0)$ and
the $W^a(\beta)$ with $\beta\in R$ representatives of the $W$-orbits in $R$
(cf. \cite{L}).

\begin{defi}\cite{C0,C}\label{H}
The trigonometric Cherednik algebra (also known as the
degenerate double affine Hecke algebra)
is the associative unital $L$-algebra $H^a_L$ satisfying:
\begin{enumerate}
\item[{\it (i)}] $H^a_L$ contains $S(V_L^*)$ and $L[W^a]$ as subalgebras.
\item[{\it (ii)}] The multiplication map defines an
isomorphism
\[S(V_L^*)\otimes_L L[W^a]\rightarrow H^a_L
\]
of $L$-modules.
\item[{\it (iii)}] The cross relations
\begin{equation}\label{cr}
s_a\cdot \xi-s_a^{(\mathbf{c})}(\xi)\cdot s_a=-\mathbf{k}_a \xi(Da^\vee),
\qquad \forall\, a\in F^a, \forall\, \xi\in V^*.
\end{equation}
\item[{\it (iv)}] $\omega\cdot \xi=\omega^{(\mathbf{c})}(\xi)\cdot \omega$
for $\omega\in\Omega$ and $\xi\in V^*$.
\end{enumerate}
\end{defi}
\begin{rema}
The cross relations for $\alpha\in F$ read
\[s_\alpha\cdot \xi-s_\alpha(\xi)\cdot s_\alpha=-\mathbf{k}_\alpha \xi(\alpha^\vee).
\]
For $a=a_0=-\theta+1$ it becomes
\[s_0\cdot\xi-(s_\theta(\xi)+\mathbf{c}\xi(\theta^\vee)1)\cdot s_0=
\mathbf{k}_0\xi(\theta^\vee).
\]
\end{rema}
The existence of $H^a_L$ needs proof; it follows from an explicit
realization of $H^a_L$ due to Cherednik \cite{C}, in terms of
Dunkl-Cherednik operators. For the sake of completeness
we will recall it in the next subsection.

In the remainder of this subsection we consider the trigonometric Cherednik
algebra with specialized parameters.
For $c\in\mathbb{C}$ and for a $W^a$-invariant function $k: R^a\rightarrow
\mathbb{C}$ (called a multiplicity function), we write $H^a(k,c)$
for the complex associative algebra obtained from $H^a_L$
by specializing $\mathbf{c}$ and $\mathbf{k}_a$ to $c$ and $k_a$,
respectively.
We call $c\in\mathbb{C}$ the level of $H^a(k,c)$.
The sub-algebra $H(k)$ of $H^a(k,c)$ generated
by $S(V^*_{\mathbb{C}})$ and $\mathbb{C}[W]$ is independent of $c$.
It is the degenerate affine Hecke algebra \cite{Dr, L}.

By induction on $l(w)$ we have
\begin{equation}\label{wxi}
w\cdot\xi=(w^{(c)}(\xi))\cdot w-\sum_{a\in R^{a,+}\cap w^{-1}R^{a,-}}k_a
\xi(Da^\vee)ws_a,\qquad \forall\, w\in W^a,\,\, \forall\, \xi\in V_{\mathbb{C}}^*
\end{equation}
in $H^a(k,c)$, cf. \cite[Prop. 1.1]{O}.
The cross relations in $H^a(k,c)$ between the simple reflections
$s_a$ ($a\in F^a$) and $p\in S(V_{\mathbb{C}}^*)$ can also be made
explicit. For this we first introduce
rescaled roots $a^{(c)}$ ($a=\alpha+m1\in R^a$) by
\[a^{(c)}=\alpha+cm1\in\hbox{Aff}(V_{\mathbb{C}}).
\]
Observe that $s_a^{(c)}=s_{a^{(c)}}$.
We also use the notation $s_{a^{(c)}}$ for the associated action
on $S(V_{\mathbb{C}}^*)\simeq P(V_{\mathbb{C}})$ by algebra automorphisms.
Observe that $w^{(c)}(a^{(c)})=(w(a))^{(c)}$
for $w\in W^a$ and $a\in R^a$.
With these notations the cross relations \eqref{cr}
in $H^a(k,c)$ imply
\begin{equation}\label{BGGform}
s_a\cdot p-s_{a^{(c)}}(p)\cdot s_a=k_a\left(\frac{s_{a^{(c)}}(p)-p}{a^{(c)}}
\right), \qquad \forall\, a\in F^a, \forall\, p\in S(V_{\mathbb{C}}^*).
\end{equation}
It follows from \eqref{BGGform} that the center
of the degenerate affine Hecke algebra $H(k)$
is the sub-algebra $S(V_{\mathbb{C}}^*)^W$ of $W$-invariant elements
in the symmetric algebra $S(V_{\mathbb{C}}^*)$ (see \cite{L} for details).

For $c\not=0$ we have $H^a(k,c)\simeq H^a(k/c,1)$ as algebras,
where $k/c$ is the multiplicity that takes value $k_a/c$ at $a\in R^a$.
The map $H^a(k,c)\rightarrow H^a(k/c,1)$ realizing the
algebra isomorphism is determined by $\xi\mapsto c\xi$ and
$w\mapsto w$ for $\xi\in V_{\mathbb{C}}^*$ and $w\in W^a$.
The center of $H^a(k,c)$ is trivial if $c\not=0$
(cf. \cite[Prop. 1.3.6]{AST} for $R$ of type $A$).

We denote $H^a(k):=H^a(k,0)$ for the trigonometric Cherednik
algebra at level $0$. Its center is studied in \cite{Ob}.
In this paper we only need the simple observation that
$S(V_{\mathbb{C}}^*)^W$ is contained in the center
of $H^a(k)$, in view of \eqref{BGGform}.
\begin{rema}
For the root system $R$ of type $A$
the spectrum of the center of $H^a(k)$ is explicitly
described in \cite{Ob}. It is called the trigonometric
Calogero-Moser space.
\end{rema}
In analogy with terminology for affine Lie algebras we call
$c=0$ the critical level (cf. \cite{Cbook, AST}).
The trigonometric Cherednik algebra
$H^a(k)$ at critical level is the main object of study in
this paper.

\subsection{The Cherednik representation}
For completeness we recall in this subsection the faithful representation
of $H^a_L$ in terms of Dunkl-Cherednik operators.
We start with two convenient lemmas for proving that
some $L[W^a]$-module $M$ and a suitable compatible family
of $L$-linear operators on $M$ give rise to a $H^a_L$-module structure
on $M$. The second lemma will
be used at a later stage with specialized parameters
to construct a representation of the trigonometric
Cherednik algebra $H^a(k)$ at critical level.
\begin{lem}\label{Dunklapproach}
Let $M$ be a left $L[W^a]$-module and $N\subseteq M$
a $L$-submodule which generates $M$ as a $L[W^a]$-module.
Suppose furthermore that
\begin{enumerate}
\item[{\bf (i)}] $T_\xi\in\textup{End}_{L}(M)$ is a family of
linear operators depending linearly on $\xi\in V^*$.
\item[{\bf (ii)}] The cross relations
\begin{equation*}
\begin{split}
s_aT_\xi-T_{s_a^{(\mathbf{c})}(\xi)}s_a
&=-\mathbf{k}_a\xi(Da^\vee)\textup{Id}_M,\qquad a\in F^a,\\
\omega T_\xi-T_{\omega^{(\mathbf{c})}(\xi)}\omega&=0,\qquad\qquad\qquad\qquad\,\,
\omega\in\Omega
\end{split}
\end{equation*}
are satisfied as endomorphisms of $M$,
where $T_{\xi+\lambda 1}:=T_\xi+\lambda\textup{Id}_M$ for
$\xi\in V^*$ and $\lambda\in L$.
\item[{\bf (iii)}] The kernel of the commutator $\lbrack T_\xi, T_\eta\rbrack$
contains $N$ for all $\xi,\eta\in V^*$.
\end{enumerate}
Then the $T_\xi$ ($\xi\in V^*$) pair-wise commute as endomorphisms
of $M$. Hence the $W^a$-action on $M$,
together with $\xi\mapsto T_\xi$,
turns $M$ into a $H^a_L$-module.
\end{lem}
\begin{proof}
The lemma is a direct consequence of the identities
\begin{equation}\label{towardscommutative}
w \lbrack T_\xi,T_{\eta}\rbrack=\lbrack T_{(Dw)\xi},T_{(Dw)\eta}\rbrack w,
\qquad \forall\, w\in W^a, \forall\, \xi,\eta\in V^*
\end{equation}
in $\textup{End}_{L}(M)$.
Formula \eqref{towardscommutative} is a consequence of properties {\bf (i)}
and {\bf (ii)} only. In fact, it suffices to establish
\eqref{towardscommutative} for
$w=s_a$ ($a\in F^a$) and for $w=\omega\in\Omega$, in which case it
follows by straightforward computations from {\bf (i)} and {\bf (ii)}.
\end{proof}
We have the following dual version of Lemma \ref{Dunklapproach}.
\begin{lem}\label{Dunklapproachdual}
Let $M$ be a left $L[W^a]$-module. Let $p: M\rightarrow N$
be a $L$-linear map to some $L$-module $N$ such that $\{0\}$ is
the only $L[W^a]$-submodule of $M$ contained in the kernel
$\textup{ker}(p)$ of $p$.

Suppose furthermore the existence of a family of
$L$-linear operators $T_\xi$ on $M$ satisfying
conditions {\bf (i)} and {\bf (ii)}
of Lemma \ref{Dunklapproach}.

If the image of the
commutator $\lbrack T_\xi, T_\eta\rbrack$ is contained in $\textup{ker}(p)$
for all $\xi,\eta\in V^*$, then the $T_\xi$ ($\xi\in V^*$)
pair-wise commute as endomorphisms
of $M$. In this situation the $W^a$-action on $M$,
together with $\xi\mapsto T_\xi$,
turns $M$ into a $H^a_L$-module.
\end{lem}
\begin{proof}
Analogous to the proof of Lemma \ref{Dunklapproach}.
\end{proof}

Lemma \ref{Dunklapproach} can be used to verify that
$H^a_L$ admits a realization in terms of Dunkl-Cherednik operators
\cite{C}.
In the present set-up it involves some small
adjustments since we do not require $\mathbf{k}_0=\mathbf{k}_\theta$.
It relates to
the extension to nonreduced root systems from \cite{O}.

We write the standard basis of the group algebra $L[Y]$
as $\{e^y\}_{y\in Y}$. The algebra structure is then governed
by $e^ye^{y^\prime}=e^{y+y^\prime}$ and $e^0=1$. Interpreting
$L[Y]$ as the algebra of regular $L$-valued functions
on $V_{\mathbb{C}}^*/2\pi\sqrt{-1}X$, the basis element $e^y$
corresponds to the trigonometic function $\xi\mapsto e^{\xi(y)}$.

\begin{defi}(cf. \cite{Ob})
The trigonometric Weyl algebra $\mathcal{A}_L$
is the unique unital associative $L$-algebra satisfying
\begin{enumerate}
\item[{\it (i)}] $\mathcal{A}_L$ contains $S(V_L^*)$
and $L[Y]$ as subalgebras.
\item[{\it (ii)}] The multiplication map defines an
isomorphism
\[ S(V_L^*)\otimes_L L[Y]\rightarrow \mathcal{A}_L
\]
of $L$-modules.
\item[{\it (iii)}] The cross relations
\[\lbrack \xi, e^y\rbrack=\mathbf{c}\xi(y)e^y
\]
for all $\xi\in V^*$ and $y\in Y$.
\end{enumerate}
\end{defi}
The indeterminates $\mathbf{k}_a$ in $\mathcal{A}_L$ are
merely dummy parameters. We include them
in the definition of $\mathcal{A}_L$ to avoid ground ring extensions
at later stages.

The existence of $\mathcal{A}_L$ is immediate, since it can be realized
as the $L$-subalgebra of $\textup{End}_L(L[Y])$ generated
by $L[Y]$ (viewed as multiplication operators) and by
$\mathbf{c}\partial_\xi$ ($\xi\in V_{\mathbb{C}}^*$), where $\partial_\xi$
is the $L$-linear derivation $\partial_\xi e^y=\xi(y)e^y$ ($y\in Y$)
of $L[Y]$.

With $\mathbf{k}$ and $\mathbf{c}$ specialized
to a fixed multiplicity function $k: R\rightarrow\mathbb{C}$
and a level $0\not=c\in\mathbb{C}$,
the associated specialized complex algebra $\mathcal{A}(c)$
is the algebra of differential operators on the compact torus
$\sqrt{-1}V^*/2\pi\sqrt{-1}X$ with regular coefficients.
For $c=0$ it
is the algebra of regular functions on the cotangent bundle
of $\sqrt{-1}V^*/2\pi\sqrt{-1}X$. It inherets the structure
of a Poisson algebra from the semiclassical limit of
$\mathcal{A}(c)$ as $c\rightarrow 0$.

Let $\mathcal{A}_L^{(\delta)}$ be the right localization
of $\mathcal{A}_L$ at $\delta:=\prod_{\alpha\in R^+}(1-e^{-2\alpha^\vee})$.
It is easy to check that $\mathcal{A}_L^{(\delta)}$ is a ring
containing $\mathcal{A}_L$.
Denote $L[Y]^{(\delta)}=\mathcal{A}_L^{(\delta)}\otimes_{\mathcal{A}_L}L[Y]$.
The Weyl group $W$ acts naturally
by $L$-algebra automorphisms on $\mathcal{A}_L$ and
$\mathcal{A}_L^{(\delta)}$.
We write $\mathcal{A}_L^{(\delta)}\#W$ for the associated
smashed product $L$-algebra. It is isomorphic
to $\mathcal{A}_L^{(\delta)}\otimes_LL[W]$
as $L$-modules. The localized $\mathcal{A}_L^{(\delta)}$-module
$L[Y]^{(\delta)}$ is a faithful $\mathcal{A}_L^{(\delta)}\#W$-module.
We call it the basic representation of $\mathcal{A}_L^{(\delta)}\#W$.

To give Cherednik's realization of $H^a_L$ as $L$-subalgebra of
$\mathcal{A}_L^{(\delta)}\#W$
it is convenient to use the following reparametrization of $\mathbf{k}$.
Define $\mathbf{l}_\alpha$ ($\alpha\in R$) by
\begin{equation*}
\mathbf{l}_\alpha=
\begin{cases}
\mathbf{k}_\alpha,\qquad &\alpha\not\in W\theta,\\
\mathbf{k}_0,\qquad &\alpha\in W\theta
\end{cases}
\end{equation*}
and set
$\rho(\mathbf{k})=\frac{1}{2}\sum_{\alpha\in R^+}\mathbf{k}_\alpha\alpha^\vee\in
V_L$.

For $\xi\in V^*$ we define the Dunkl-Cherednik operator \cite{C,O}
by
\begin{equation}\label{DCoper}
\mathbf{D}_\xi:= \xi+\sum_{\alpha\in R^+}\xi(\alpha^\vee)
\left(\frac{\mathbf{k}_\alpha+\mathbf{l}_\alpha e^{-\alpha^\vee}}
{1-e^{-2\alpha^\vee}}\right)(1-s_\alpha)-\xi(\rho(\mathbf{k}))\in
\mathcal{A}_L^{(\delta)}\#W.
\end{equation}
Under the basic representation of $\mathcal{A}_L^{(\delta)}\#W$ it restricts
to an operator on $L[Y]$.

If $\theta(Y)=\mathbb{Z}$ then
$a_0\in W^a\theta$, hence $\mathbf{k}_0=\mathbf{k}_\theta$ and
$\mathbf{l}_\alpha=\mathbf{k}_\alpha$
for all $\alpha\in R$. In this case
\[\mathbf{D}_\xi=\xi+\sum_{\alpha\in R^+}\mathbf{k}_\alpha\xi(\alpha^\vee)
\frac{1}{1-e^{-\alpha^\vee}}(1-s_\alpha)-\xi(\rho(\mathbf{k})),
\]
which is the Dunkl-Cherednik operator associated to
the reduced root system $R$ (see \cite{C}).
If $\theta(Y)=2\mathbb{Z}$ then $R$ is of type $C_n$ for some $n\geq 1$.
In this case $\mathbf{D}_\xi$ is the Dunkl-Cherednik
operator associated to the nonreduced root system of type $\textup{BC}_n$
(see \cite{O}).

We now give Cherednik's
well known realization of $H^a_L$ as subalgebra of
$\mathcal{A}_L^{(\delta)}\#W$. It ensures that the
$L$-algebra $H^a_L$, as defined in Definition \ref{H}, exists.
\begin{thm}\label{Cherrepr}\cite{C}
The assignments
\begin{equation}\label{assign}
\begin{split}
\xi&\mapsto \mathbf{D}_\xi,\qquad \xi\in V^*,\\
w&\mapsto w,\qquad\,\,\, w\in W,\\
t_y&\mapsto e^y,\qquad\,\, y\in Y
\end{split}
\end{equation}
uniquely extend to an injective $L$-algebra homomorphism
$H^a_L\rightarrow \mathcal{A}_L^{(\delta)}\#W$.
\end{thm}
\begin{proof}
We sketch a proof based on Lemma \ref{Dunklapproach}.
Composing with the basic representation of $\mathcal{A}_L^{(\delta)}\#W$ we view
the right hand sides of \eqref{assign} as elements in
$\textup{End}_L(L[Y])$. The verification that the assignments
extend to an $L$-algebra homomorphism then reduces to the cross relations by
Lemma \ref{Dunklapproach} applied to the $L[W^a]$-module $L[Y]$,
$T_\xi=\mathbf{D}_\xi$ ($\xi\in V^*$) and $N=Le^0$. The cross relations
can be verified by direct computations (cf. \cite{OL}).
Injectivity follows by a standard argument.
\end{proof}
We write
$p(\mathbf{D})$ for the element in
$\mathcal{A}_L^{(\delta)}\#W$ corresponding to $p\in S(V_L^*)$.

Specializing $\mathbf{k}$ to a multiplicity function
$k: R\rightarrow\mathbb{C}$
and the level $\mathbf{c}$ to a noncritical value $0\not=c\in\mathbb{C}$,
Theorem \ref{Cherrepr} gives rise to the faithful Cherednik representation
of $H^a(k,c)$ on $\mathbb{C}[Y]$ in which $\xi\in V_{\mathbb{C}}^*$
acts by the Dunkl-Cherednik operator
\[D_\xi:=c\partial_\xi+\sum_{\alpha\in R^+}\xi(\alpha^\vee)
\left(\frac{k_\alpha+l_\alpha e^{-\alpha^\vee}}
{1-e^{-2\alpha^\vee}}\right)(1-s_\alpha)-\xi(\rho(k))\in
\textup{End}_{\mathbb{C}}(\mathbb{C}[Y])
\]
(with the obvious meaning of $l_\alpha$ and $\rho(k)$).
The corresponding differential-reflection operators $p(D)$
restrict to endomorphisms of $\mathbb{C}[Y]^W$ if
$p\in S(V_{\mathbb{C}}^*)^W$, in which case it acts as a
differential operator. The resulting commuting differential
operators $\{p(D)\}_{p\in S(V_{\mathbb{C}}^*)^W}$ on $\mathbb{C}[Y]^W$ are,
up to a gauge factor, the conserved quantum integrals
of the quantum trigonometric Calogero-Moser
system associated to $R$ (respectively the nonreduced root system of type
$\textup{BC}_n$)
if $\theta(Y)=\mathbb{Z}$ (respectively $\theta(Y)=2\mathbb{Z}$), see
\cite{OComp} and \cite[Part I \S 2.2]{HS}.
Specialized at critical level $c=0$ the
$\{p(D)\}_{p\in S(V_{\mathbb{C}}^*)^W}$ relate to the classical conserved
integrals of the trigonometric Calogero-Moser system.

At critical level $c=0$ various other representations of $H^a(k)$
are known that involve Dunkl type operators,
see e.g. \cite{C0, Cbook, EOS}.
In \cite{C0, Cbook} Dunkl operators with infinite
reflection terms are used
(it relates to quantum elliptic Calogero-Moser systems).
In \cite{EOS} representations of $H^a(k)$ are considered that involve
Dunkl type operators involving jumps over the affine root hyperplanes.
It relates to quantum integrable Calogero-Moser type systems with delta
function potentials. In Section \ref{basic} we 
generalize the latter representations.
It allows us
to include quantum spin-particle systems with delta function
potentials in the present framework.
In these cases the action of the
commutative subgroup $Y$ of $W^a$ is by translations.
This is a key difference to Theorem \ref{Cherrepr},
where $Y$ acts by multiplication operators.


\subsection{Intertwiners}

The algebra $H^a(k,c)$ gives rise to a large supply of
nontrivial $W^a$-cocycles. The construction uses the
normalized intertwiners associated to $H^a(k,c)$.

Let $H^a(k,c)_{loc}$ be the localized
trigonometric Cherednik algebra obtained
by right-adjoining the inverses of $a^{(c)}+k_a$ to $H^a(k,c)$ for all
$a\in R^a$. Denote $H^a(k,c)_{loc}^\times$ for the group
of units in $H^a(k,c)_{loc}$.
\begin{prop}\cite{C}\label{intertwiners}
There exists a unique group homomorphism
$W^a\rightarrow H^a(k,c)_{loc}^\times$, denoted by $w\mapsto I_w^{k,c}$,
satisfying
\begin{equation*}
\begin{split}
I_{s_a}^{k,c}&=(s_a\cdot a^{(c)}+k_a)\cdot (a^{(c)}-k_a)^{-1},\qquad a\in F^a,\\
I_{\omega}^{k,c}&=\omega,\qquad\qquad\quad\qquad\qquad\qquad\qquad
\quad \omega\in\Omega.
\end{split}
\end{equation*}
Furthermore,
\[I_w^{k,c}\cdot p=w^{(c)}(p)\cdot I_w^{k,c}\qquad\forall\, w\in W^a,
\forall\, p\in S(V_{\mathbb{C}}^*)
\]
in $H^a(k,c)_{loc}$.
\end{prop}
The $I_w^{k,c}$ ($w\in W^a$) are called the normalized
intertwiners associated to $H^a(k,c)$. Observe that
the $I_w^{k,c}$ for $w\in W$ are independent of the level $c$.

Consider the set $\mathcal{S}_{k,c}$
consisting of $t\in V_{\mathbb{C}}$
satisfying $a^{(c)}(t)\not=k_a$ for all $a\in R^a$. Note that
$S_{k,c}$ is invariant for the action $t\mapsto
w^{(c)}(t)$ of $W^a$ on $V_{\mathbb{C}}$. For
$t\in\mathcal{S}_{k,c}$ consider the character
\[\chi_t: S(V_{\mathbb{C}}^*)_{loc}\rightarrow \mathbb{C},\qquad
p\mapsto p(t),
\]
where $S(V_{\mathbb{C}}^*)_{loc}$ is the localized algebra obtained by adjoining
the inverses of $a^c+k_a$ to $S(V_{\mathbb{C}}^*)_{loc}$ for all $a\in R^a$
(which canonically is a sub-algebra of $H^a(k,c)_{loc}$).
The map $w\mapsto w\otimes_{\chi_t}1$ for $w\in W^a$
gives a vector space identification between the group algebra $\mathbb{C}[W^a]$
and the principal $H^a(k,c)_{loc}$-module
$M(t):=\hbox{Ind}_{S(V_{\mathbb{C}}^*)_{loc}}^{H^a(k,c)_{loc}}(\chi_t)$.
For $w\in W^a$ and $t\in\mathcal{S}_{k,c}$
denote $I_{w}^{k,c}(t)$
for the element in the group algebra $\mathbb{C}[W^a]$ associated to
$I_w^{k,c}\otimes_{\chi_t}1\in M(\lambda)$. In other words,
$I_w^{k,c}(t)=\sum_{v\in W^a}p_v^w(t)v$ if the normalized intertwiner
$I_w^{k,c}$ expands as $I_w^{k,c}=\sum_{v\in W^a}v\cdot p_v^w$
($p_v^w\in S(V_{\mathbb{C}}^*)_{loc}$) in $H^a(k,c)_{loc}$.
The $I_w^{k,c}(t)$ satisfy (and are uniquely characterized by)
\begin{equation}\label{cocyclecondition}
\begin{split}
I_{s_a}^{k,c}(t)&=\frac{a^{(c)}(t)s_a+k_a}{a^{(c)}(t)-k_a},\\
I_{\omega}^{k,c}(t)&=\omega,\\
I_{\sigma\tau}^{k,c}(t)&=I_{\sigma}^{k,c}(\tau^{(c)}(t))
I_\tau^{k,c}(t)
\end{split}
\end{equation}
for $a\in F^a$, $\omega\in\Omega$ and $\sigma,\tau\in W^a$.
\begin{rema}\label{YBE}
The cocycle $\{I_w^{k,c}(t)\}_{w\in W^a}$ gives rise to unitary solutions of
generalized Yang-Baxter equations with spectral parameters
(see \cite{Chercoc} for the general theory)
and plays a key role in the description of the Bethe ansatz
equations associated to quantum spin-particle systems with delta function
interactions. For the latter application one is forced to consider
the cocycle at critical level $c=0$. The basic example is
for $R$ of type $A_{n-1}$ and
$X=\bigoplus_{j=1}^n\mathbb{Z}\epsilon_j$, in which case we have
discussed these observations in detail in the introduction.
One of the main goals in the present
paper is to give similar interpretations of the
cocycle $\{I_w^{k,0}(t)\}_{w\in W^a}$ for arbitrary affine Weyl groups
$W^a$.
\end{rema}
Denote
\[i_w^{k,c}(t)=\chi\bigl(I_w^{k,c}(t)\bigr),\qquad
w\in W^a,
\]
where $\chi: \mathbb{C}[W^a]\rightarrow\mathbb{C}$
is the algebra homomorphism mapping $w$ to $1$ for all $w\in W^a$.
By induction on the length $l(w)$ of $w\in W^a$ we have
\begin{equation}\label{explicitc}
i_w^{k,c}(t)=\prod_{a\in R^{a,+}\cap w^{-1}R^{a,-}}
\frac{a^{(c)}(t)+k_a}{a^{(c)}(t)-k_a},\qquad w\in W^a.
\end{equation}
In particular,
\[
i_w^{k,c}(t)=\prod_{\alpha\in R^+\cap w^{-1}R^-}\frac{\alpha(t)+k_\alpha}
{\alpha(t)-k_\alpha},\qquad w\in W,
\]
which is independent of $c$.
At critical level $c=0$ we write
$I_w^k(t)=I_w^{k,0}(t)$ and $i_w^k(t)=
i_w^{k,0}(t)$. We furthermore set
$I_y^k(t)=I_{t_y}^k(t)$ and $i_y^k(t)=i_{t_y}^k(t)$
for $y\in Y$.
Let $\mathbb{C}[W^a]^\times$ be the group of units in $\mathbb{C}[W^a]$.
\begin{prop}\label{explicitcocycle}
Let $t\in\mathcal{S}_k:=\mathcal{S}_{k,0}$.\\
{\bf (i)} The map
\[y\mapsto I_y^k(t)
\]
defines a group homomorphism $Y\rightarrow \mathbb{C}[W^a]^\times$.\\
{\bf (ii)} Suppose that $\theta(Y)=\mathbb{Z}$. Then
\begin{equation}\label{ikred}
i_y^k(t)=\prod_{\alpha\in R^+}\left(\frac{\alpha(t)+k_\alpha}
{\alpha(t)-k_\alpha}\right)^{\alpha(y)}
\end{equation}
for $y\in Y$.\\
{\bf (iii)} Suppose that $\theta(Y)=2\mathbb{Z}$. Then
\begin{equation}\label{iknonred}
i_y^k(t)=\prod_{\stackrel{\alpha\in R^+:}{\alpha\not\in W\theta}}
\left(\frac{\alpha(t)+k_\alpha}{\alpha(t)-k_\alpha}\right)^{\alpha(y)}
\prod_{\stackrel{\beta\in R^+:}{\beta\in W\theta}}
\left(\frac{\beta(t)+k_\theta}
{\beta(t)-k_\theta}\right)^{\frac{\beta(y)}{2}}
\left(\frac{\beta(t)+k_0}{\beta(t)-k_0}\right)^{\frac{\beta(y)}{2}}
\end{equation}
for $y\in Y$.
\end{prop}
\begin{proof}
{\bf (i)}. This follows from the cocycle property of $I_w^k(t)$
(the last identity of \eqref{cocyclecondition}) since translations
$t_y$ ($y\in Y$) act trivially under the action $w\mapsto w^{(0)}=Dw$ of $W^a$
on $V_{\mathbb{C}}$.\\
{\bf (ii)} \& {\bf (iii)}. For $y\in Y$ we have
\[
R^{a,+}\cap t_y^{-1}R^{a,-}=\{\alpha+m1 \,\, | \,\, \alpha\in R,\,\,\,
B(\alpha)\leq m<B(\alpha)+\alpha(y)\}
\]
where $B(\alpha)=1$ if $\alpha\in R^-$ and $=0$ if $\alpha\in R^+$.
Using the convention that $\prod_{r=l}^mc_r=1$ if $l>m$ we can thus write
\begin{equation}\label{one}
i_y^k(t)=\prod_{\alpha\in R^+}\prod_{m=0}^{\alpha(y)-1}
\frac{\alpha(t)+k_{\alpha+m1}}{\alpha(t)-k_{\alpha+m1}}
\prod_{m=1}^{-\alpha(y)}\frac{\alpha(t)-k_{-\alpha+m1}}
{\alpha(t)+k_{-\alpha+m1}}.
\end{equation}
If $\theta(Y)=\mathbb{Z}$ then $k_a=k_{Da}$ for all $a\in R^a$ and
\eqref{one} reduces to \eqref{ikred}.

If $\theta(Y)=2\mathbb{Z}$ then $k_a=k_{Da}$ for $a\in R^a$ with
$Da\not\in W\theta$. Furthermore, for $\alpha\in W\theta$
we have $k_{\alpha+(2m)1}=k_\theta$ and
$k_{\alpha+(2m+1)1}=k_0$ for $m\in\mathbb{Z}$.
Formula \eqref{one} then becomes \eqref{iknonred} after
straightforward computations.
\end{proof}


\section{Representations of the trigonometric
Cherednik algebra at critical level}\label{basic}

We fix an $W^a$-invariant multiplicity function 
$k: R^a\rightarrow \mathbb{C}$ throughout this section.

\subsection{The algebra $A(k)$}\label{newform}

In the next subsection we give the representation of
the trigonometric Cherednik algebra $H^a(k)$ at critical level 
in terms of vector-valued Dunkl-type operators.
To avoid the use of twisted affine root systems it is convenient
to work with an adjusted presentation of $H^a(k)$, 
which we give now first.

\begin{defi}
Let $A(k)$ be the unital associative algebra over $\mathbb{C}$
satisfying:
\begin{enumerate}
\item[{\it (i)}] $A(k)$ contains $S(V_{\mathbb{C}})$
and $\mathbb{C}[W^a]$ as subalgebras.
\item[{\it (ii)}] The multiplication map defines an isomorphism
\[
S(V_{\mathbb{C}})\otimes_{\mathbb{C}}\mathbb{C}[W^a]\rightarrow A(k).
\]
\item[{\it (iii)}] The cross relations 
\[s_a\cdot v-s_{Da}(v)\cdot s_a=-k_aDa(v),
\qquad \forall\, a\in F^a
\]
for all $v\in V_{\mathbb{C}}$.
\item[{\it (iv)}] $\omega\cdot v=(D\omega)(v)\cdot\omega$
for $\omega\in\Omega$ and $v\in V_{\mathbb{C}}$.
\end{enumerate}
\end{defi}

The cross relations {\it (iii)} may be replaced by
\begin{equation}\label{DemazureLusztig}
s_a\cdot p- s_{Da}(p)\cdot s_a=k_a\Delta_{Da}(p)\qquad
\forall\, a\in F^a, \forall\, p\in S(V_{\mathbb{C}}),
\end{equation}
where the divided difference operator
$\Delta_\alpha: S(V_{\mathbb{C}})\rightarrow S(V_{\mathbb{C}})$ ($\alpha\in R$)
is given by
\[\Delta_\alpha(p)=\frac{s_\alpha(p)-p}{\alpha^\vee},\qquad p\in S(V_{\mathbb{C}}).
\]
Induction to the length of $w\in W^a$ also proves the commutation relations
\begin{equation}\label{forintertwiner}
w\cdot v=((Dw)v)\cdot w-\sum_{a\in R^{a,+}\cap w^{-1}R^{a,-}}k_a(Da)(v)ws_a
\qquad \forall\, w\in W^a, \forall\, v\in V
\end{equation}
in $A(k)$ (cf. \eqref{wxi}). The algebra $A(k)$ is the 
trigonometric Cherednik algebra at critical level, 
as follows from the following
lemma.
\begin{lem}
Let $k^\vee$ be the multiplicity function
$k^\vee_a=2k_a/\langle Da,Da\rangle$ ($a\in R^a$). 
The assignments $v\mapsto \langle v,\cdot\rangle\in V_{\mathbb{C}}^*$
and $w\mapsto w$ for $v\in V_{\mathbb{C}}$ and $w\in W^a$ uniquely
extend to a unital algebra isomorphism $A(k^\vee)\rightarrow H^a(k)$.
\end{lem}
\begin{proof}
Straightforward check.
\end{proof}
It is convenient to alter the notations for 
the cocycles $I_w^k(t)\in \mathbb{C}[W^a]$ ($w\in W^a$) accordingly.
It results in the following definitions and formulas.

Let $C_k\subset V_{\mathbb{C}}^*$ be the $W$-invariant
set of vectors $\lambda\in V_{\mathbb{C}}^*$
satisfying $\lambda(Da^\vee)\not=k_a$ for all $a\in R^a$.
For $\lambda\in C_k$ there exists unique $J_w^k(\lambda)\in\mathbb{C}[W^a]$
($w\in W^a$) satisfying 
\begin{equation*}
\begin{split}
J_{s_a}^k(\lambda)&=\frac{\lambda(Da^\vee)s_a+k_a}{\lambda(Da^\vee)-k_a},
\qquad\qquad a\in F^a,\\
J_\omega^k(\lambda)&=\omega,\qquad\qquad\qquad\qquad\qquad \omega\in\Omega,\\
J_{\sigma\tau}^k(\lambda)&=J_\sigma^k((D\tau)\lambda)J_\tau^k(\lambda),
\qquad\quad\,\, \sigma,\tau\in W^a.
\end{split}
\end{equation*}
Denoting $j_w^k(\lambda)=\chi(J_w^k(\lambda))$ ($w\in W^a$)
and $j_y^k(\lambda)=j_{t_y}^k(\lambda)$ ($y\in Y$), 
we have 
\begin{equation}\label{jfiniteexplicit}
j_w^k(\lambda)=\prod_{\alpha\in R^+\cap w^{-1}R^-}
\frac{\lambda(\alpha^\vee)+k_\alpha}{\lambda(\alpha^\vee)-k_\alpha},\qquad
w\in W
\end{equation}
and 
\begin{equation}\label{jexplicit}
j_{y}^k(\lambda)=
\prod_{\stackrel{\alpha\in R^+:}{\alpha\not\in W\theta}}
\left(\frac{\lambda(\alpha^\vee)+k_\alpha}
{\lambda(\alpha^\vee)-k_\alpha}\right)^{\alpha(y)}
\prod_{\stackrel{\beta\in R^+:}{\beta\in W\theta}}
\left(\frac{\lambda(\beta^\vee)+k_\theta}
{\lambda(\beta^\vee)-k_\theta}\right)^{\frac{\beta(y)}{2}}
\left(\frac{\lambda(\beta^\vee)+k_0}
{\lambda(\beta^\vee)-k_0}\right)^{\frac{\beta(y)}{2}}
\end{equation}
for $y\in Y$. If $\theta(Y)=\mathbb{Z}$ then the latter formula simplifies to
\[j_{y}^k(\lambda)=\prod_{\alpha\in R^+}\left(\frac{\lambda(\alpha^\vee)+k_\alpha}
{\lambda(\alpha^\vee)-k_\alpha}\right)^{\alpha(y)},\qquad y\in Y.
\]

\subsection{The Dunkl type operators}

For a complex associative algebra $A$ we write $\textup{Mod}_A$
for the category of complex left $A$-modules. The 
category $\textup{Mod}_{\mathbb{C}[W^a]}$ is a tensor category
with unit object the trivial $W^a$-module, which we denote by $\mathbb{I}$.
In this subsection we define an explicit functor 
$F_{dr}^k:\textup{Mod}_{\mathbb{C}[W^a]}\rightarrow
\textup{Mod}_{A(k)}$ using vector valued Dunkl type operators
(the subindex ``dr'' stands for ``differential-reflection'').
It extends results from the paper \cite{EOS}, in which the $A(k)$-module
$F_{dr}^k(\mathbb{I})$ was constructed. 
The Dunkl type differential-reflection operators will 
have the special feature that at the chamber $w^{-1}C_+$ ($w\in W^a$)
the number of occuring reflection terms is equal to the 
``distance'' $l(w)$ of $w^{-1}C_+$ to the fundamental chamber $C_+$.
Special cases and other examples of such differential-reflection operators 
were considered in \cite{P,MW,HK,H,EOS}.

Recall that $V_a=a^{-1}(0)$ is the affine root hyperplane of the
affine root $a\in R^a$. Denote 
$V_{reg}=V\setminus \cup_{a\in R^{a,+}}V_a$ for the
set of regular elements in $V$. It is well known that
\[V_{reg}=\bigcup_{w\in W\ltimes Q^\vee}w(C_+)
\]
(disjoint union), with $C_+\subset V_{reg}$ given by
\[C_+=\{v\in V \, | \, a(v)>0 \quad \forall\, a\in F^a\}.
\]
Furthermore, the subgroup $\Omega$ of length zero elements in 
$W^a$ permutes $F^a$, hence it acts on the fundamental domain $C_+$.
In particular, $W^a$ permutes the connected components
$\mathcal{C}=\{w(C_+) \, | \, w\in W\ltimes Q^\vee\}$ of $V_{reg}$. 
We call $C\in \mathcal{C}$ a chamber, and $C_+$ the fundamental chamber.

Denote $C^\omega(V)$ be the space of complex-valued, real analytic
functions on $V$. For a complex left $W^a$-module $M$ we now define
a suitable space of $M$-valued functions on $V$ which are 
real analytic on $V_{reg}$, but which are ``fuzzy'' 
on the affine root hyperplanes, in the sense that we do not specify
its values on the affine root hyperplanes (cf. Remark \ref{multi}{\bf (ii)}).

\begin{defi}
Let $M$ be a complex left $W^a$-module.
We write $B^\omega(V;M)$ for the complex vector space of functions
$f: V_{reg}\rightarrow M$ satisfying, for all $C\in\mathcal{C}$, 
$f|_C=f_C|_C$ for some $f_C\in C^\omega(V)\otimes_{\mathbb{C}}M$ (algebraic 
tensor product).
\end{defi}
\begin{rema}\label{multi}
{\bf (i)} The map $f\mapsto (f_C)_{C\in\mathcal{C}}$
defines a complex linear isomorphism $B^\omega(V;M)\rightarrow
\prod_{C\in\mathcal{C}}(C^\omega(V)\otimes M)$. We will use
this identification without further reference.\\
{\bf (ii)} A function $f\in B^\omega(V;M)$ can be interpreted as 
a multi $M$-valued function on $V$ by defining
\[f(v)=\{f_C(v)\}_{C\in\mathcal{C}:\, v\in\overline{C}}
\]
for any $v\in V$, 
where $\overline{C}$ is the closure of the chamber 
$C$ in the Euclidean space $V$.
\end{rema}
The space $B^\omega(V;M)$ is a $W^a$-module by
\[\bigl(w\cdot f\bigr)(v)=w\bigl(f(w^{-1}(v))\bigr),
\qquad w\in W^a,\,\, f\in B^\omega(V;M),\,\, v\in V_{reg}.
\]
We denote the action by a dot to avoid confusion with the 
$W^a$-action on $M$. Viewing the $f_C$'s as $M$-valued
functions on $V$, the action can be expressed as $(w\cdot f)_C(v)=
w\bigl(f_{w^{-1}C}(w^{-1}v)\bigr)$ for $w\in W^a$, $C\in\mathcal{C}$
and $v\in V$. Alternatively it can be expressed
as $(w\cdot f)_C=(w\otimes w)f_{w^{-1}C}$, viewed
as identity in $C^\omega(V)\otimes_{\mathbb{C}}M$. Here we use 
the natural $W^a$-action
on $C^\omega(V)$, given by $(wg)(v):=g(w^{-1}v)$ for $g\in C^\omega(V)$, 
$w\in W^a$ and $v\in V$.

Let $\mathcal{I}: \mathbb{R}^\times\rightarrow \{0,1\}$ be the 
indicator function 
of $\mathbb{R}_{<0}$ and write $\mathcal{I}_a(v):=\mathcal{I}(a(v))$
for $a\in R^{a,+}$ and $v\in V_{reg}$. Then
\begin{equation*}
\mathcal{I}_a|_{w^{-1}C_+}\equiv
\begin{cases}
1 \quad &\hbox{ if } wa\in R^{a,-},\\
0 \quad &\hbox{ if } wa\in R^{a,+}.
\end{cases}
\end{equation*}
In particular, for $w\in W^a$ we have
$\mathcal{I}_a|_{w^{-1}C_+}\equiv 1$ only if 
$a$ is a positive affine root from the finite 
set $R^{a,+}\cap w^{-1}R^{a,-}$. 
This ensures that the Dunkl type operator
\begin{equation}\label{Dunkl}
\mathcal{D}_v^{k,M}f=\partial_v f-\sum_{a\in R^{a,+}}k_a(Da)(v)
\mathcal{I}_a(\cdot)(s_a\cdot f),\qquad f\in B^\omega(V;M)
\end{equation}
for $v\in V$ defines a well defined linear operator on 
$B^\omega(V;M)$, where 
\[(\partial_vf)(v^\prime)=\left. \frac{d}{dt}\right|_{t=0}f(v^\prime+tv)
\] 
for $v^\prime\in V_{reg}$ is the directional derivative of $f$ 
in the direction $v\in V$. Indeed, on a fixed 
chamber $w^{-1}C_+\in\mathcal{C}$ ($w\in W^a$)
the formula \eqref{Dunkl} for the Dunkl operator
gives
\begin{equation}\label{Dunklperchamber}
(\mathcal{D}_v^{k,M}f)_{w^{-1}C_+}=\partial_vf_{w^{-1}C_+}-
\sum_{a\in R^{a,+}\cap w^{-1}R^{a,-}}
k_a(Da)(v)(s_a\otimes s_a)f_{s_aw^{-1}C_+}
\end{equation}
as identity in $C^\omega(V)\otimes_{\mathbb{C}}M$. 

\begin{thm}
Let $M$ be a $W^a$-module and $k$ an $W^a$-invariant
multiplicity function on $R^a$. The assignments
\begin{equation*}
\begin{split}
v&\mapsto \mathcal{D}_v^{k,M},\qquad v\in V,\\
w&\mapsto w\cdot,\qquad\,\,\,\,\,\,\, w\in W^a
\end{split}
\end{equation*}
uniquely extend to an algebra
homomorphism $\pi_{k,M}: A(k)\rightarrow 
\textup{End}_{\mathbb{C}}\bigl(B^\omega(V;M)\bigr)$.
\end{thm}
\begin{proof}
We apply Lemma \ref{Dunklapproachdual} (with adjusted notations and 
specialized parameters)
to $B^\omega(V;M)$, considered as $W^a$-module by the dot-action. 

A direct computation (compare with \cite[Thm. 4.1]{EOS}) shows that
the $\mathcal{D}_v^{k,M}$ ($v\in V$) satisfy the $A(k)$ type cross relations
with respect to the dot-action of $W^a$ on $B^\omega(V;M)$. 
It thus remains to construct an appropriate complex vector space $N$ and a
linear map $p: B^\omega(V;M)\rightarrow N$ satisfying the conditions of
Lemma \ref{Dunklapproachdual}. We take $N=C^\omega(V)\otimes_{\mathbb{C}}M$
and $p$ the linear map $p\bigl((f_C)_{C\in\mathcal{C}}\bigr)=f_{C_+}$.
Clearly the only $W^a$-submodule of $B^\omega(V;M)$ that is contained in 
$\textup{ker}(p)$ is $\{0\}$. Furthermore,
\[
p\bigl(\mathcal{D}_v^{k,M}f\bigr)=\partial_vp(f),\qquad \forall\, 
f\in B^\omega(V;M)
\]
by \eqref{Dunklperchamber}, hence the image of $\lbrack \mathcal{D}_v^{k,M}, 
\mathcal{D}_{v^\prime}^{k,M}\rbrack$ is contained in $\textup{ker}(p)$.
Thus Lemma \ref{Dunklapproachdual} can be applied. 
It yields the desired result.
\end{proof}
\begin{rema}
{\bf (i)} The theorem reduces to \cite[Thm. 4.2]{EOS} when 
$Y=Q^\vee$, $k_0=k_\theta$ and $M=\mathbb{I}$.\\
{\bf (ii)} The representation $\pi_{k,\mathbb{I}}$ 
is faithful, compare with (the proof of) \cite[Thm. 4.2]{EOS}.
\end{rema}
\begin{cor}\label{FunctorG}
{\bf (i)}  The assignment
$M\mapsto B(V;M)$ defines a covariant functor 
$F_{dr}^k: \textup{Mod}_{\mathbb{C}[W^a]}\rightarrow \textup{Mod}_{A(k)}$
(with the obvious definition on morphisms).\\
{\bf (ii)} For $p\in S(V_{\mathbb{C}})^W$ we have $\pi_{k,M}(p)=
p(\partial)\otimes\textup{Id}_M$,
where $p(\partial)$ is the constant-coefficient differential operator 
associated to $p$. 
\end{cor}
\begin{proof}
{\bf (i)} Clear.\\
{\bf (ii)} 
This is analogous to the proof of
\cite[Cor. 4.6]{EOS}. 
\end{proof}


\subsection{Integral-reflection operators}

In this subsection we define another explicit functor 
$F_{ir}^k:\textup{Mod}_{\mathbb{C}[W^a]}\rightarrow
\textup{Mod}_{A(k)}$ using integral-reflection operators \cite{GS,G}
(the subindex ``ir'' stands for ``integral-reflection'').
It again extends results from the paper \cite{EOS}, in which the $A(k)$-module
$F_{ir}^k(\mathbb{I})$ was constructed.

Let $M$ be a complex left $W^a$-module. The linear dual
$M^*=\hbox{Hom}_{\mathbb{C}}(M;\mathbb{C})$ is a left $W^a$-module
by $(w\psi)(m)=\psi(w^{-1}m)$ for $w\in W^a$, $\psi\in M^*$ and $m\in M$.
Consider the $A(k)$-module $\hbox{Ind}_{\mathbb{C}[W^a]}^{A(k)}
\bigl(M^*\bigr)$. As complex vector
spaces, we have 
\[\hbox{Ind}_{\mathbb{C}[W^a]}^{A(k)}\bigl(M^*\bigr)\simeq
S(V_{\mathbb{C}})\otimes_{\mathbb{C}}M^*.
\]
Expressing the action through the linear
isomorphism we get the following explicit $A(k)$-action
on $S(V_{\mathbb{C}})\otimes_{\mathbb{C}}M^*$, 
\begin{equation}\label{almost}
\begin{split}
s_a(p\otimes \psi)&=s_{Da}(p)\otimes s_a \psi+
k_a\Delta_{Da}(p)\otimes \psi,
\qquad a\in F^a,\\
\omega(p\otimes \psi)&=(D\omega)(p)\otimes\omega \psi,
\qquad\qquad\qquad\qquad\quad \omega\in\Omega,\\
r(p\otimes \psi)&=(rp)\otimes \psi,\qquad\quad\qquad\qquad\qquad\qquad\,\,
r\in S(V_{\mathbb{C}})
\end{split}
\end{equation}
for $p\in S(V_{\mathbb{C}})$ and $\psi\in M^*$.
We now endow the linear dual $\bigl(S(V_{\mathbb{C}})\otimes M^*)^*$
with the structure of left $A(k)$-module using the following simple lemma. 
\begin{lem}
For a complex left $A(k)$-module $N$, the linear dual $N^*$
is a left $A(k)$-module by
\[(X\psi)(n)=\psi(X^\dagger n),\qquad \psi\in N^*,
X\in A(k), n\in N,
\]
where $X\mapsto X^\dagger$ is the unique unital complex linear 
anti-algebra involution of $A(k)$ satisfying $w^\dagger=w^{-1}$ ($w\in W^a$) 
and $v^\dagger=v$ ($v\in V$).
\end{lem}
Next we rewrite the $A(k)$-action on a suitable subspace
of $\bigl(S(V^*)\otimes_{\mathbb{C}}M^*\bigr)^*$ in terms of
integral-reflection operators.

For $a\in R^a$ we define the integral operator $I(a)$ on $C^\omega(V)$ by
\begin{equation}
(I(a)f)(v)=\int_0^{a(v)}f(v-tDa^\vee)dt,\qquad
f\in C^\omega(V),\,\, v\in V.
\end{equation}
With respect to the natural action $(wf)(v)=f(w^{-1}(v))$
of $w\in W^a$ on $f\in C^\omega(V)$ the integral operators 
satisfy
\begin{equation}\label{invariance}
wI(a)w^{-1}=I(w(a)),\qquad \forall\, w\in W^a, \forall\, a\in R^a.
\end{equation} 
The integral operators $I(\alpha)$ are adjoint to
the divided difference operator $-\Delta_\alpha$ ($\alpha\in R$)
in the following sense.

\begin{lem}\cite{G2}\label{G}
Let $\bigl(\cdot,\cdot\bigr): S(V_{\mathbb{C}})\times C^\omega(V)\rightarrow
\mathbb{C}$ be the non-degenerate complex bilinear form defined by
\[\bigl(p,f\bigr)=\bigl(p(\partial)f\bigr)(0),\qquad\quad
p\in S(V_{\mathbb{C}}), f\in C^\omega(V).
\]
Then
\[\bigl(\Delta_\alpha(p),f\bigr)=
-\bigl(p,I(\alpha)f\bigr)
\]
for $\alpha\in R$, $p\in S(V_{\mathbb{C}})$ and $f\in C^\omega(V)$. 
\end{lem}
We obtain the following immediate consequence.
\begin{cor}\label{step}
Let $M$ be a left $W^a$-module.
Consider $C^\omega(V)\otimes_{\mathbb{C}}M$ as linear
subspace of $\bigl(S(V_{\mathbb{C}})\otimes_{\mathbb{C}}M^*\bigr)^*$
by interpreting $f\otimes m\in C^\omega(V)\otimes_{\mathbb{C}}M$
as the linear functional
\[p\otimes \psi\mapsto \bigl(p,f\bigr)\psi(m),
\qquad\quad p\in S(V_{\mathbb{C}}), \psi\in M^*.
\]
Then $C^\omega(V)\otimes_{\mathbb{C}}M$ is a $A(k)$-submodule
of $\bigl(S(V_{\mathbb{C}})\otimes_{\mathbb{C}}M^*\bigr)^*$. The corresponding
left $A(k)$-action on $C^\omega(V)\otimes_{\mathbb{C}}M$ is explicitly given by
\begin{equation*}
\begin{split}
r&\mapsto r(\partial)\otimes\textup{Id}_M, \,\,\qquad
\qquad\qquad\qquad\qquad r\in S(V_{\mathbb{C}}),\\
s_a&\mapsto s_{Da}\otimes s_a-k_aI(Da)\otimes\textup{Id}_M, 
\qquad\quad\,\,\, a\in F^a,\\
\omega&\mapsto D\omega\otimes\omega,\qquad\qquad\qquad\qquad\qquad\quad\,\,\,\,
\omega\in\Omega. 
\end{split}
\end{equation*}
\end{cor}
\begin{proof}
Chasing the actions, the proof easily reduces to Lemma \ref{G}
and the obvious identities
\begin{equation*}
\begin{split}
\bigl(rp,f\bigr)&=\bigl(p,r(\partial)f\bigr), \quad
\qquad\,\, r\in S(V_{\mathbb{C}}),\\
\bigl(w(p),f\bigr)&=\bigl(p,w^{-1}(f)\bigr),
\qquad\,\,\, w\in W
\end{split}
\end{equation*}
for $p\in S(V_{\mathbb{C}})$ and $f\in C^\omega(V)$. 
\end{proof}

Note that the action of $W^a$ on the first tensor leg
$C^\omega(V)$ of the $A(k)$-module $C^\omega(V)\otimes_{\mathbb{C}}M$ 
is the pull-back action of $W$ under the gradient map $D$ 
(the commutative subgroup $Y$ of $W^a$ acts trivially).
We now upgrade it to the standard action of $W^a$ on $C^\omega(V)$.

\begin{thm}\label{intrefl}
Let $M$ be a complex left $W^a$-module and $k$ a multiplicity
function on $R^a$. The assignments 
\begin{equation*}
\begin{split}
v&\mapsto \partial_v\otimes\textup{Id}_M, 
\,\,\qquad\qquad\qquad\qquad v\in V,\\
s_a&\mapsto s_a\otimes s_a-k_aI(a)\otimes\textup{Id}_M,\qquad\,\, 
a\in F^a,\\
\omega&\mapsto \omega\otimes\omega, \qquad\qquad\qquad\qquad\quad\,\,\,\,\, 
\omega\in\Omega 
\end{split}
\end{equation*}
uniquely extend to an algebra homomorphism 
$Q_{k,M}: A(k)\rightarrow 
\textup{End}_{\mathbb{C}}\bigl(C^\omega(V)\otimes_{\mathbb{C}}M\bigr)$.
\end{thm}
\begin{proof} 
Denote $Q: A(k)\rightarrow
\textup{End}_{\mathbb{C}}\bigl(C^\omega(V)\otimes_{\mathbb{C}}M\bigr)$
for the representation map associated to the $A(k)$-action of 
Corollary \ref{step}. Since the explicit assignment $Q_{k,M}$
in the statement of the theorem satisfies
$Q_{k,M}(v)=Q(v)$ and 
$Q_{k,M}(s_\alpha)=Q(s_\alpha)$ 
for $v\in V$ and $\alpha\in F$, it remains to verify the following
identities.
\begin{enumerate}
\item[{\it (i)}] $Q_{k,M}(s_0)^2=\hbox{Id}$.
\item[{\it (ii)}] If $(s_{\alpha}s_0)^m=1$
in $W^a$ for some $\alpha\in F$ and $m\in\mathbb{N}$, then
\[\bigl(Q_{k,M}(s_\alpha)Q_{k,M}(s_0)\bigr)^m=\hbox{Id}.
\]
\item[{\it (iii)}] $Q_{k,M}(\omega)Q_{k,M}(s_a)=Q_{k,M}(s_{\omega(a)})
Q_{k,M}(\omega)$ for $a\in F^a$ and $\omega\in\Omega$.
\item[{\it (iv)}] The cross relations
\[Q_{k,M}(s_0)(\partial_v\otimes \textup{Id}_M)-
(\partial_{s_\theta(v)}\otimes\textup{Id}_M)Q_{k,M}(s_0)=k_0\theta(v)
\]
for $v\in V$.
\item[{\it (v)}] $Q_{k,M}(\omega)(\partial_v\otimes \hbox{Id}_M)=
(\partial_{D\omega(v)}\otimes\hbox{Id}_M)Q_{k,M}(\omega)$
for $\omega\in\Omega$ and $v\in V$.
\end{enumerate}
The identities {\it (i)}-{\it (iii)} show that $Q_{k,M}$ 
defines a $W^a$-action on $C^\omega(V)\otimes_{\mathbb{C}}M$. 
The identities {\it (iv)} and {\it (v)} ensure that
also all the necessary cross relations are satisfied in order for
$Q_{k,M}$ to extend to an algebra homomorphism
$Q_{k,M}: A(k)\rightarrow \textup{End}_{\mathbb{C}}
\bigl(C^\omega(V)\otimes_{\mathbb{C}}M\bigr)$. 

The proofs of {\it (i)}--{\it (v)} are much facilitated by the 
simple observation that
\begin{equation}\label{h}
Q_{k,M}(s_0)=(t_{u}\otimes \hbox{Id}_M)Q(s_{0})
(t_{u}^{-1}\otimes \hbox{Id}_M)
\end{equation}
for any $u\in V$ such that $\theta(u)=1$ (in which case
$t_{u}(-\theta)=-\theta+1=a_0$).
By \eqref{h} and Corollary \ref{step}
the identities {\it (i)}, {\it (iv)} and {\it (v)}
are immediate. For the braid relation {\it (ii)}, observe that
$(s_{\alpha}s_0)^m=1$ in $W^a$ for some $\alpha\in F$ and 
$m\in\mathbb{N}$ implies that $\alpha$ is not a scalar
multiple of $\theta$ in $V^*$. Hence there exists a vector $u\in V$ 
such that $\alpha(u)=0$ and $\theta(u)=1$.
In this case we have, besides \eqref{h}, 
\[Q_{k,M}(s_\alpha)=(t_{u}\otimes\hbox{Id}_M)Q(s_\alpha)
(t_{u}^{-1}\otimes \hbox{Id}_M).
\]
Hence $\bigl(Q_{k,M}(s_\alpha)Q_{k,M}(s_0)\bigr)^m=\textup{Id}$
follows by conjugating the corresponding valid identity for $Q$
by $(t_u\otimes\hbox{Id}_M$).

The identities {\it (iii)} can be checked by a direct computation.
\end{proof}
\begin{rema}\label{FunctorF}
{\bf (i)}
For $k_0=k_\theta$, $Y=Q^\vee$ and 
$M=\mathbb{I}$, Theorem \ref{intrefl} reduces to \cite[Thm. 5.2]{EOS}.\\
{\bf (ii)} The assignment $M\mapsto C^\omega(V)\otimes_{\mathbb{C}}M$
defines a covariant functor $F_{ir}^k: \textup{Mod}_{\mathbb{C}[W^a]}\rightarrow
\textup{Mod}_{A(k)}$ (with the obvious definition on morphisms).
\end{rema}

Let $\textup{Forg}^k: \textup{Mod}_{A(k)}\rightarrow 
\textup{Mod}_{\mathbb{C}[W^a]}$
be the forgetful functor. In the following proposition we compare
the space $F_{ir}^k(M)^{W^a}$ of $W^a$-invariants of the
$W^a$-module $\textup{Forg}^k\bigl(F_{ir}^k(M)\bigr)$ with the space 
$\bigl(F_{ir}^k(\mathbb{I})\otimes M\bigr)^{W^a}$ of $W^a$-invariants
of the tensor product
$W^a$-module $\textup{Forg}^k\bigl(F_{ir}^k(\mathbb{I})\bigr)\otimes M$:
\begin{prop}\label{invariantsequal}
Let $M$ be a left $W^a$-module. Then
\[F_{ir}^k(M)^{W^a}=\bigl(F_{ir}^k(\mathbb{I})\otimes M\bigr)^{W^a}.
\]
\end{prop}
\begin{proof}
It is convenient to set out some notations first.
We denote $\pi_M$ for the representation map of $M$.
The space $C^\omega(V)$ will be considered with respect to two
different left $W^a$-actions. The first is the 
regular $W^a$-action
\begin{equation}\label{regularaction}
(L(w)g)(v)=g(w^{-1}v),\qquad g\in C^\omega(V),\,\, v\in V,\,\, w\in W^a.
\end{equation}
The resulting $W^a$-module will still be denoted by $C^\omega(V)$. 
The other $W^a$-action is 
$Q_{k,\mathbb{I}}|_{W^a}$, in which case we denote the $W^a$-module
by $\textup{Forg}^k\bigl(F_{ir}^k(\mathbb{I})\bigr)$. We omit
the forgetful functor from the notations from now on.
 
Observe that both $W^a$-modules $F_{ir}^k(M)$ and
$F_{ir}^k(\mathbb{I})\otimes M$
have the same underlying vector space 
$C^\omega(V)\otimes_{\mathbb{C}}M$. Their subspaces of 
$\Omega$-invariants coincide trivially.

Fix $a\in F^a$. For a left $W^a$-module $N$ we write $N_+$ (respectively
$N_-$) for the space of $s_a$-invariants (respectively $s_a$-antiinvariants)
in $N$. It suffices to show that 
\begin{equation}\label{equivalence}
F_{ir}^k(M)_+=\bigl(F_{ir}^k(\mathbb{I})\otimes M\bigr)_+.
\end{equation}

By Theorem \ref{intrefl} we have
\[Q_{k,M}(s_a)=Q_{k,\mathbb{I}}(s_a)\otimes\pi_M(s_a)+R_M(s_a)
\]
with $R_M(s_a)=k_aI(a)\otimes\bigl(\pi_M(s_a)-\textup{Id}_M)$. We claim
that
\begin{equation}\label{Restaction}
\begin{split}
R_M(s_a)\bigl((F_{ir}^k(\mathbb{I})\otimes M)_+\bigr)&=\{0\},\\
R_M(s_a)\bigl((F_{ir}^k(\mathbb{I})\otimes M)_-\bigr)&\subseteq 
(F_{ir}^k(\mathbb{I})\otimes M)_+.
\end{split}
\end{equation}
To prove \eqref{Restaction}, we first note that
\[I(a)\bigl(C^\omega(V)_-\bigr)=\{0\},\qquad
I(a)\bigl(C^\omega(V)_+\bigr)\subseteq C^\omega(V)_-,
\]
see, e.g., \cite[Lemma 3.4]{EOS} for the first part. This lifts
to the (anti)invariants with respect to the $k$-dependent actions,
\[I(a)\bigl(F_{ir}^k(\mathbb{I})_-\bigr)=\{0\},\qquad
I(a)\bigl(F_{ir}^k(\mathbb{I})_+\bigr)\subseteq F_{ir}^k(\mathbb{I})_-,
\]
since $Q_{k,\mathbb{I}}(s_a)g=\pm g$ implies $k_aI(a)g=(L(s_a)\mp \textup{Id})g$. 
This in turn implies \eqref{Restaction}. 

We are now ready to prove \eqref{equivalence}.
The inclusion $\supseteq$ of \eqref{equivalence} is an immediate
consequence of \eqref{Restaction}. For the converse inclusion
we take $f\in F_{ir}^k(M)_+$ and write $f=f_++f_-$ with
$f_{\pm}\in \bigl(F_{ir}^k(\mathbb{I})\otimes M\bigr)_{\pm}$. 
By \eqref{Restaction} we have
\[f=Q_{k,M}(s_a)f=(f_++R_M(s_a)f_-)-f_-
\]
hence $f_+=f_++R_M(s_a)f_-$ and $f_-=-f_-$, again by \eqref{Restaction}.
We conclude that $f_-=0$ and $f=f_+\in\bigl(F_{ir}^k(\mathbb{I})\otimes M)_+$. 
\end{proof}


\subsection{The propagation transformation}\label{propag}
Let $M$ be a left $W^a$-module.
The $W^a$-action on $C^\omega(V)\otimes_{\mathbb{C}}M$ 
in terms of integral-reflection
operators (Theorem \ref{intrefl}) can be used to propagate a plane wave
attached to the fundamental chamber $C_+$ to a common eigenfunction of the
Dunkl operators $\mathcal{D}_v^{k,M}$ ($v\in V$). 
This idea goes back to Gutkin and Sutherland \cite{GS,G} 
and was further explored
in \cite[\S 5]{EOS} for $M=\mathbb{I}$. 
In the present set-up it will give rise to a natural 
transformation $T^k: F_{ir}^k\rightarrow F_{dr}^k$ 
between the two functors 
$F_{ir}^k,F_{dr}^k: \textup{Mod}_{\mathbb{C}[W^a]}\rightarrow 
\textup{Mod}_{A(k)}$ (see Remark \ref{FunctorF}{\bf (ii)} and
Corollary \ref{FunctorG}{\bf (i)}). 

We now first define $T^{k,M}$ for a given $W^a$-module $M$ as a complex
linear map.
\begin{lem}
Let $M$ be a left $W^a$-module.
There exists a unique linear map
\[T^{k,M}: C^\omega(V)\otimes_{\mathbb{C}}M\rightarrow B^\omega(V;M)
\]
satisfying $\bigl(T^{k,M}f\bigr)_{C_+}=f$ and $T^{k,M}(Q_{k,M}(w)f)=w\cdot 
(T^{k,M}f)$ for all $w\in W^a$.
\end{lem}
\begin{proof}
The required properties of $T^{k,M}$ 
can equivalently be formulated as 
\[
(T^{k,M}f)_{w^{-1}C_+}=(w^{-1}\otimes w^{-1})(Q_{k,M}(w)f)\qquad
\forall w\in W^a
\]
in $C^\omega(V)\otimes_{\mathbb{C}}M$. The
existence and uniqueness are now immediate.
\end{proof}
By construction the $T^{k,M}$ defines a natural transformation
\[
T^k: \textup{Forg}^k\circ F_{ir}^k\rightarrow \textup{Forg}^k\circ F_{dr}^k.
\]
We call $T^{k}$ the propagation transformation.
\begin{prop}
$T^k$ defines a natural transformation $T^k: F_{ir}^k\rightarrow F_{dr}^k$.
\end{prop}
\begin{proof}
It suffices to show that 
\[T^{k,M}\circ\partial_v=\mathcal{D}_v^{k,M}\circ T^{k,M}
\]
for all $v\in V$. This follows from \eqref{forintertwiner} by
a direct computation.
\end{proof}

For $M$ a $W^a$-module the
corresponding morphism $T^{k,M}: F_{ir}^k(M)\rightarrow F_{dr}^k(M)$ 
in $\textup{Mod}_{A(k)}$
is a monomorphism. We denote its image by $C^{\omega,k}(V;M)$ and
write $F_{dr}^{\omega,k}: \textup{Mod}_{\mathbb{C}[W^a]}\rightarrow
\textup{Mod}_{A(k)}$ for the associated functor. On objects it is given 
by $F_{dr}^{\omega,k}(M)=C^{\omega,k}(V;M)$. 
The propagation transformation $T^k$ 
defines an equivalence $T^k: F_{ir}^k\overset{\sim}{\longrightarrow} 
F_{dr}^{\omega,k}$ between the two functors 
$F_{ir}^k, F_{dr}^{\omega,k}: \textup{Mod}_{\mathbb{C}[W^a]}\rightarrow
\textup{Mod}_{A(k)}$.

We now characterize
$C^{\omega,k}(V;M)$ as subspace 
of $B^\omega(V;M)$ in terms of derivative jump conditions
over affine root hyperplanes.

For a chamber $C=wC_+$ ($w\in W^a$)
we write $R_C^a=\{wa \, | \, a\in F^a\}$. The 
$V_b\cap \overline{C}$ ($b\in R_C^a$) are the walls of the chamber
$C$. In particular, if $b\in R_C^a$ then 
$C$ and $s_bC$ are adjacent chambers 
with common wall $V_b\cap\overline{C}$, and $Db^\vee$ is 
a vector normal to $V_b\cap\overline{C}$ which points towards $C$.

\begin{prop}\label{jumpprop}
Let $M$ be a $W^a$-module. Then $C^{\omega,k}(V;M)$
is the space of functions $f\in B^\omega(V;M)$ satisfying
$\forall\, C\in\mathcal{C}$, $\forall\, b\in R^a_C$, $\forall\, v\in V_b\cap
\overline{C}$,
\begin{equation}\label{jumpconditions}
\bigl(p(\partial)f_{C}\bigr)(v)-\bigl(p(\partial)f_{s_bC}\bigr)(v)=
k_bs_b\lbrack\bigl((\Delta_{Db}p)(\partial)f_C\bigr)(v)\rbrack
\end{equation}
for all $p\in S(V_{\mathbb{C}})$. Here we use the $s_b$-action 
on $M$ in the right hand side of \eqref{jumpconditions}.
\end{prop}
\begin{proof}
The proof is essentially the same as the proof of \cite[Thm. 5.3]{EOS},
which deals with the case that  $Y=Q^\vee$, $k_0=k_\theta$
and $M=\mathbb{I}$. We thus only give
a sketch of the proof in the present set-up.

We call a vector $v\in V$ subregular if it lies
on exactly one affine root hyperplane $V_a$ ($a\in R^{a,+}$). 
The $Q_{k,M}$-image of the cross relations \eqref{DemazureLusztig}
imply that $f\in C^{\omega,k}(V;M)$ satisfies the jump conditions
\eqref{jumpconditions} for subregular $v\in V_b\cap\overline{C}$.
By continuity it then holds for all $v\in V_b\cap\overline{C}$.

Suppose on the other hand 
that $f\in B^{\omega}(V;M)$ satisfies the jump conditions
\eqref{jumpconditions} for all $p\in S(V_{\mathbb{C}})$.
For $p=1$ the jump conditions \eqref{jumpconditions}
are an elaborate way of saying that $f$ is a continuous 
$M$-valued function on $V$ (cf. Remark \ref{multi}{\bf (ii)}). 
The jump conditions \eqref{jumpconditions} also imply that $f$
satisfies normal derivative jump conditions of all orders
over the walls (see Remark \ref{jump}). 
Since the $f_C$'s are real analytic it follows that 
$f$ is uniquely determined by its restriction to the fundamental chamber
$C_+$.
Hence $T^{k,M}(f_{C_+})=f\in C^{\omega,k}(V;M)$.
\end{proof}
\begin{cor}
If $f\in C^{\omega,k}(V;M)$ then $\forall\, C\in \mathcal{C}$,
$\forall\, b\in R_C^a$, $\forall\, v\in V_b\cap\overline{C}$,
\[(p(\partial)f_C)(v)=(p(\partial)f_{s_bC})(v)
\]
for all $p=\sum_{m\geq 0}(Db^\vee)^{2m}p_m\in S(V_{\mathbb{C}})$ with
$p_m\in S(V_{Db})_{\mathbb{C}}$.
\end{cor}

\begin{rema}\label{jump}
{\bf (i)} In the right hand side of the jump conditions
\eqref{jumpconditions} we may replace
$f_C$ by $f_{s_bC}$ (or by $\frac{1}{2}(f_C+f_{s_bC})$).\\
{\bf (ii)} Take $C\in \mathcal{C}$, $b\in R_C^a$ and 
$v\in V_b\cap\overline{C}$. Recall that 
the vector $Db^\vee$ is normal to the wall $V_b\cap \overline{C}$ 
of $C$ and points towards the chamber $C$.
Then \eqref{jumpconditions} for $p=(Db^\vee)^r$ ($r\in\mathbb{N}$)
become the following $r$th order normal derivative
jump condition of $f$ over the wall $V_b\cap \overline{C}$ at $v$,
\begin{equation}\label{jumpconditionshigherorder}
(\partial_{Db^\vee}^rf_C)(v)-(\partial_{Db^\vee}^rf_{s_bC})(v)=
((-1)^r-1)k_bs_b\lbrack (\partial_{Db^\vee}^{r-1}f_C)(v)\rbrack.
\end{equation}
For $r=1$ it reduces to
\begin{equation}\label{jumpconditionsdegree1}
(\partial_{Db^\vee}f_C)(v)-(\partial_{Db^\vee}f_{s_bC})(v)=
-2k_bs_b\bigl(f_C(v)\bigr).
\end{equation}
Continuity and the higher order normal derivative jump conditions
\eqref{jumpconditionshigherorder} 
suffice to characterize $C^\omega(V;M)$ as subspace of $B^\omega(V;M)$,
see the proof of Proposition \ref{jumpprop}.
\end{rema}

\subsection{Relation to quantum many body problems}

We fix in this subsection a left $W^a$-module $M$. By
Corollary \ref{FunctorG}{\bf (ii)} the constant coefficient
differential operators
$\pi_{k,M}(p)=p(\partial)$ ($p\in S(V_{\mathbb{C}})^W$) act on $C^{\omega,k}(V;M)$.
They can be interpreted as 
quantum conserved integrals of a quantum system with delta function potentials 
as follows.

Write $C(V;M)$ for the space
of continuous $M$-valued functions on $V$: it consists of functions
$f: V\rightarrow M$ such that for all $v\in V$ there exists an open
neighborhood $U$ of $v$ in $V$ such that $f|_U\in C(U)\otimes_{\mathbb{C}}M$.
Note that $C^{\omega,k}(V;M)\subset C(V;M)$. 
Denote $C_c^\infty(V)$ for the smooth, compactly supported, complex
valued functions on $V$. 
Write $dv$ for the Euclidean volume measure
on $V$. We also write $dv$ 
for the induced volume measure on the affine root hyperplanes $V_b$
($b\in R^{a,+}$). We have a linear embedding
$\iota: C(V;M)\rightarrow \textup{Hom}_{\mathbb{C}}\bigl(C_c^\infty(V);M)$
defined by
\begin{equation}\label{weak}
(\iota f)(\phi):=\int_Vf(v)\phi(v)dv,\qquad f\in C(V;M),\,\, 
\phi\in C_c^{\infty}(V).
\end{equation}
This allows us to view $\pi_{k,M}(p)f$ 
for $f\in CB^\omega(V;M):=C(V;M)\cap B^\omega(V;M)$ and for 
$p\in S(V_{\mathbb{C}})$ as the $M$-valued distribution 
$\iota(\pi_{k,M}(p)f)$. We occasionally omit $\iota$ if it is
clear that the weak interpretation is meant.

Denote $\|\cdot\|$ for the norm on the Euclidean space $V$.
We write
\begin{equation}\label{Hamiltonian}
\mathcal{H}_k^M=-\Delta-\sum_{b\in R^{a}}\frac{k_b}{\|Db^\vee\|}
\delta(b(\cdot))s_b
\end{equation}
for the linear map $\mathcal{H}_k^M: C(V;M)\rightarrow 
\textup{Hom}_{\mathbb{C}}\bigl(C_c^\infty(V);M)$ defined by
\[
(\mathcal{H}_k^Mf)(\phi)=-\int_Vf(v)(\Delta\phi)(v)dv
-\sum_{b\in R^a}\frac{k_b}{\|Db^\vee\|}
\int_{V_b}s_b(f(v))\phi(v)dv
\]
for $f\in C(V;M)$ and $\phi\in C^{\infty}_c(V)$, where
$s_b$ only acts on $M$. 

We also write $\|\cdot\|$ for the norm of the Euclidean space $V^*$
and interpret the $W$-invariant polynomial $\|\cdot\|^2$ on $V^*$
as element in $S(V_{\mathbb{C}})^W$ in the usual way. 
Observe that $\|\partial\|^2$ is the Laplacean $\Delta$ on $V$.
Note furthermore that $\pi_{k,M}(\|\cdot\|^2)=\Delta$ by Corollary
\ref{FunctorG}{\bf (ii)}. The following result directly implies
a reformulation of the spectral problem for $\mathcal{H}_k^M$
as a boundary value problem (cf. Subsection \ref{SpTh}).
\begin{prop}\label{BVP}
Let $M$ be a left $W^a$-module and $f\in CB^\omega(V;M)$.
Then
\[-\iota\bigl(\Delta f\bigr)=\mathcal{H}_k^Mf
\]
if and only if $f$ satisfies the derivative jump conditions
\eqref{jumpconditions} for $p\in S(V_{\mathbb{C}})$ of degree one.
\end{prop}
\begin{proof}
Let $\phi\in C^{\infty}_c(V)$ and $f\in CB^\omega(V;M)$.
Then
\[\iota\bigl(\Delta f\bigr)(\phi)=
\sum_{C\in\mathcal{C}}\int_C(\Delta f)(v)\phi(v)dv
\]
(only finitely many terms contribute to the sum). Green's identity 
allows us to rewrite the right hand side as
\[\int_Vf(v)(\Delta\phi)(v)dv+
\sum_L\int_L\bigl((\partial_{n^L}f_{C_{n^L}})(v)-
(\partial_{n^L}f_{C^\prime_{n^L}})(v)\bigr)\phi(v)dv.
\]
Here the sum runs over all the walls $L$,
$n^L\in V$ is a unit normal vector to $L$ and $C_{n^L}$
(respectively $C_{n^L}^\prime$) is the chamber with wall $L$
such that $n^L$ is pointing away (respectively towards)
the chamber. 

On the other hand, if $f\in CB^\omega(V;M)$ satisfies the jump conditions
\eqref{jumpconditions} for $p\in S(V_{\mathbb{C}})$ of degree one
iff, for a given wall $L=V_b\cap\overline{C}$
($C\in\mathcal{C}$, $b\in R_C^a$),
\[(\partial_{n^L}f_{C_{n^L}})(v)-
(\partial_{n^L}f_{C^\prime_{n^L}})(v)=
\frac{2k_b}{\|Db^\vee\|}s_b(f(v))
\]
for $v\in L=V_b\cap\overline{C}$ (see \eqref{jumpconditionsdegree1}). 
The result follows now directly.
\end{proof}
It is natural to consider $\mathcal{H}_k^{M}$ as the quantum Hamiltonian
of a quantum physical system. 
Particular cases of these quantum systems have been extensively studied,
see, e.g., \cite{LL,YY,McG,Y,Ga,GS,G,GY,MW,HO,EOS}, to name just a few.
Most studies in the literature deal with the root system
$R$ of type $A$, in which case the quantum system 
describes one dimensional 
quantum spin-particles with pair-wise delta 
function interactions. The assumption $k_a<0$
(respectively $k_a>0$) then corresponds to
repulsive (respectively attractive) delta function interactions.
For root systems $R$ of classical type the quantum system 
relates to one dimensional quantum
spin-particles with pair-wise delta function interactions and 
boundary reflection terms.

\section{Spectral theory of the quantum many body problem}
\label{Spectralsection}
Let $N$ be a left $A(k)$-module and $\lambda\in V_{\mathbb{C}}^*$.
Recall that $S(V_{\mathbb{C}})^W$ is part of the center of $A(k)$.
We denote $N_\lambda\subseteq N$ for the $A(k)$-submodule
\[
N_\lambda=\{ n\in N \,\, | \,\, p\cdot n=p(\lambda)n\quad \forall\,
p\in S(V_{\mathbb{C}})^W\}.
\]
We call $N_\lambda$ the $A(k)$-submodule of $N$ with central character
$\lambda$. In this section we study the
$A(k)$-submodule $F_{ir}^k(M)_\lambda\simeq F_{dr}^{\omega,k}(M)_\lambda$.


\subsection{The spectral problem}\label{SpTh}

For $\lambda\in V_{\mathbb{C}}^*$ we write 
\begin{equation}\label{Espace}
E(\lambda)=\{f\in C^\omega(V) \,\, | \,\, p(\partial)f=p(\lambda)f
\quad \forall\, p\in S(V_{\mathbb{C}})^W \}.
\end{equation}
Viewed as $W$-module with the natural $W$-action (cf. \eqref{regularaction}), 
$E(\lambda)$ is isomorphic to
the regular $W$-representation, see \cite{St}. If 
$\lambda(\alpha^\vee)\not=0$ for all $\alpha\in R$ then
$E(\lambda)=\bigoplus_{w\in W}\mathbb{C}e^{w\lambda}$
with $e^\lambda$ the complex plane wave $v\mapsto e^{\lambda(v)}$.

We fix a left $W^a$-module $M$. The subalgebra $S(V_{\mathbb{C}})\subset A(k)$
acts by constant coefficient differential operators on
the $A(k)$-module $F_{ir}^k(M)$, hence
\[
F_{ir}^k(M)_\lambda=E(\lambda)\otimes_{\mathbb{C}}M
\]
as vector spaces.
Consider now the $A(k)$-modules $F_{dr}^{\omega,k}(M)_\lambda\subset
F_{dr}^k(M)_\lambda$. By Corollary \ref{FunctorG}{\bf (ii)},
\[
F_{dr}^k(M)_\lambda=\{f\in B^\omega(V;M) \,\, | \,\, f_C\in 
E(\lambda)\otimes_{\mathbb{C}}M\quad \forall\, C\in\mathcal{C} \}
\]
as vector spaces.
In subsection \ref{propag} we characterized 
$F_{dr}^{\omega,k}(M)$ as vector subspace
of $F_{dr}^k(M)$ in terms of derivative
jump conditions of arbitrary order over the affine root hyperplanes.
Restricted to the submodules of central character $\lambda$,
the derivative jump conditions of order $\leq 1$ suffice:
\begin{prop}\label{elliptic}
Suppose $f\in F_{dr}^k(M)_\lambda$ satisfies the derivative jump conditions
\eqref{jumpconditions} for $p\in S(V_{\mathbb{C}})$ of degree $\leq 1$. Then
$f\in F_{dr}^{\omega,k}(M)_\lambda$.
\end{prop}
\begin{proof}
Choose $f\in F_{dr}^k(M)_\lambda$ satisfying \eqref{jumpconditions}
for $p\in S(V_{\mathbb{C}})$ of degree $\leq 1$. Then
\[
\psi:=T^{k,M}(f_{C_+})-f\in F_{dr}^k(M)_\lambda
\]
also satisfies \eqref{jumpconditions} for $p\in S(V_{\mathbb{C}})$ of 
degree $\leq 1$, and $\psi_{C_+}\equiv 0$. 
It suffices to show that $\psi=0$.
Recall that the neighboring chambers of a given chamber $C\in\mathcal{C}$
are given by $s_bC$ ($b\in R_C^a$). To prove that $\psi=0$
it suffices to show that $\psi_C\equiv 0$ implies 
$\psi_{s_bC}\equiv 0$ ($b\in R_C^a$).

Choose a chamber $C\in\mathcal{C}$ such that $\psi_{C}\equiv 0$
and choose an affine root $b\in R_C^a$. Write $L_b$ for the set of subregular
vectors $v$ on the wall $V_b\cap\overline{C}$. 
The subset $U=C\cup s_bC\cup L_b$ of $V$ is open and pathwise
connected.
Since $\psi$ satisfies \eqref{jumpconditions} for 
$p=1$ we have that $\psi|_U\in C(U)\otimes M$. We 
choose an arbitrary linear
functional $\chi\in M^*$ and write $\Psi=\chi(\psi|_U)\in C(U)$.
It suffices to show that $\Psi\equiv 0$. 

We have $\Psi|_{C\cup L_b}\equiv 0$ by assumption.
On the other hand, $\Psi|_{s_bC\cup L_b}$
is the restriction to $s_bC\cup L_b$ of some $g\in E(\lambda)$,
since $\psi\in F_{dr}^k(M)_\lambda$. 
In particular, $\Delta\Psi=\|\lambda\|^2\Psi$ on $C\cup s_bC$.
Since $\psi$ satisfies the jump conditions 
\eqref{jumpconditions} for $p$ of degree one, 
all directional derivatives of $\Psi$ at the $v\in L_b$ vanish.
Hence $\Psi$ is a weak eigenfunction of $\Delta$ on $U$ with eigenvalue
$\|\lambda\|^2$. This forces $\Psi$ to be smooth on $U$.
Consequently $g$ is zero in an open neighborhood of $v\in L_b$.
Hence $g\equiv 0$ and $\Psi\equiv 0$, as desired.
\end{proof}
Denote $CB^\omega(V;M)_\lambda=C(V;M)\cap F^k_{dr}(M)_\lambda$
for $\lambda\in V_{\mathbb{C}}^*$. Recall that $\|\lambda\|^2$
is the evaluation of the polynomial $\|\cdot\|^2\in S(V_{\mathbb{C}})^W$
at $\lambda$.
\begin{thm}\label{ellcor}
Let $f\in CB^\omega(V;M)_\lambda$.
Then $\mathcal{H}_k^M f=-\|\lambda\|^2f$ weakly if and only if
$f=T^{k,M}(f_{C_+})\in F_{dr}^{\omega,k}(M)_\lambda$.
\end{thm}
\begin{proof}
Suppose $f\in CB^\omega(V;M)_\lambda$ satisfies 
$\mathcal{H}_k^Mf=-\|\lambda\|^2f$ weakly.
By Proposition \ref{BVP}, $f$ satisfies the jump 
conditions \eqref{jumpconditions}
for $p\in S(V_{\mathbb{C}})$ of degree $\leq 1$. Hence 
$f=T^{k,M}(f_{C_+})\in F_{dr}^{\omega,k}(M)$ by
Proposition \ref{elliptic}. For the converse, if
$f\in F_{dr}^{\omega,k}(M)_\lambda$ then 
\[\mathcal{H}_k^Mf=-\Delta f=-\|\lambda\|^2f
\]
weakly, where the first equality is due to Proposition \ref{BVP}.
\end{proof}
These results lead to the following concrete procedure to
produce solutions to the spectral problem of the quantum many
body problem (it is analogous to the usual Bethe ansatz
methods for the one dimensional quantum Bose gas with delta
function interactions and its root system
generalizations, see, e.g.,
\cite{LL,McG,Ga,GS,G} and the introduction):
pick $g\in E(\lambda)\otimes_{\mathbb{C}}M$ (which is a solution to the
spectral problem for the free Hamiltonians $p(\partial)$ 
($p\in S(V_{\mathbb{C}})^W$)). Then Theorem \ref{ellcor}
shows that $f=(f_C)_{C\in\mathcal{C}}=T^{k,M}g$ 
is the unique solution of 
\[\mathcal{H}_k^Mf=-\|\lambda\|^2f
\]
such that $f_{C_+}=g$ and such that 
the $f_C\in E(\lambda)\otimes_{\mathbb{C}}M$
satisfy the derivative 
jump conditions \eqref{jumpconditions}
for $b\in V_{b}\cap \overline{C}$ and for $p\in
S(V_{\mathbb{C}})$ of degree $\leq 1$. For $R$ of type $A$,
this gives Theorem \ref{introTHM} from the introduction.
Furthermore, $f$ satisfies the derivative jump conditions
\eqref{jumpconditions} for all 
$p\in S(V_{\mathbb{C}})$, and the explicit
formula for the propagation operator $T^{k,M}$ in terms of integral
reflection operators tells us how to explicitly construct the $f_C$:
\[ 
f_{w^{-1}C_+}=(w^{-1}\otimes w^{-1})Q_{k,M}(w)g,\qquad \forall\, w\in W^a.
\]

This is not the end of the story though! For $R\subset\mathbb{R}^n$ 
of type $A_{n-1}$ with $X=\bigoplus_{i=1}^n\mathbb{Z}\epsilon_i$
the quantum Hamiltonian $\mathcal{H}_k^{M}$ can be interpreted weakly
on $V/Y$. If we take furthermore $M=P^{\otimes n}$ for some complex vector space
$P$ with the $S_n$ action the permutations of the tensor entries and with the
trivial action of $Y$, $\mathcal{H}_k^{M}$ represents the quantum Hamiltonian of
a quantum system describing $n$ quantum spin-particles on the circle 
$S^1=\mathbb{R}/\mathbb{Z}$ with internal spins (the quantum spin state
space of the quantum particles is $P$) 
and with pair-wise delta function interactions. 
The associated spectral problem amounts to analyzing the subspace
$F_{dr}^{\omega,k}(M)_\lambda^Y$ of $Y$-translation invariant functions
in $F_{dr}^{\omega,k}(M)_\lambda$. In this case the class of functions 
$g\in E(\lambda)\otimes_{\mathbb{C}}M$ such that the associated 
$f=T^{k,M}g\in F_{dr}^{\omega,k}(M)_\lambda$ is $Y$-translation invariant
is more subtle to characterize. 

In the present general set-up we will give various characterizations
of the subspace $F_{dr}^{\omega,k}(M)_\lambda^{W^a}$ of $W^a=W\ltimes Y$
invariants in $F_{dr}^{\omega,k}(M)_\lambda$ (in the context of the previous
paragraph this corresponds to the bosonic $Y$-translation invariant
theory). The following type of characterization is commonly used in the physics
literature on one dimensional 
quantum spin-particle systems with delta function interaction
(see, e.g., \cite{LL,McG,Y,LLP,GY,CY,AFK}). 

\begin{prop}\label{BAEdirect}
Let $f\in F_{dr}^{\omega,k}(M)_\lambda$. 
Then $f$ is $W^a$-invariant if and only if
$g:=f_{C_+}\in E(\lambda)\otimes_{\mathbb{C}}M$ satisfies 
$\forall\, a\in F^a$, $\forall\, v\in V_a\cap\overline{C_+}$,
$\forall\, \omega\in\Omega$,
\begin{equation}\label{jsimple}
\begin{split}
s_a(g(v))&=g(v),\\
(\textup{Id}_M+s_a)\bigl((\partial_{Da^\vee}g)(v)\bigr)&=-2k_ag(v),\\
(\omega\otimes\omega)g&=g.
\end{split}
\end{equation}
\end{prop}
\begin{proof}
Suppose that $f$ is $W^a$-invariant. 
Then $f_{s_aC_+}=(s_a\otimes s_a)f_{C_+}$  
($a\in F^a$). For $b\in F^a$ and $v\in V_b\cap\overline{C_+}$ the jump condition
\eqref{jumpconditions} with $p=1$ becomes the first line of 
\eqref{jsimple}. The jump condition 
\eqref{jumpconditionsdegree1} for $b\in F^a$
becomes the second line of \eqref{jsimple}.
Let $\omega\in \Omega$. Since $\omega(C_+)=C_+$ we have
$f_{C_+}=(\omega\otimes\omega)f_{C_+}$, which is the third line 
of \eqref{jsimple}.

Conversely, let $f\in F_{dr}^{\omega,k}(M)_\lambda$ with $g:=f_{C_+}$ satisfying
\eqref{jsimple}. Then $g\in E(\lambda)\otimes_{\mathbb{C}}M$ and $f=T^{k,M}g$.
Define $h=(h_C)_{C\in \mathcal{C}}\in F_{dr}^k(M)_\lambda$ by 
\[
h_{w^{-1}C_+}=(w^{-1}\otimes w^{-1})g,\qquad \forall\, w\in W^a.
\]
This is well defined by the third line of \eqref{jsimple}.
By construction $h$ is $W^a$-invariant
and $h_{C_+}=g=f_{C_+}$. We now show that $h$ 
satisfies the derivative jump conditions
\eqref{jumpconditions} for $p\in S(V_{\mathbb{C}})$ of degree $\leq 1$.
By the $W^a$-invariance of $h$ it suffices to verify the jump conditions
for $C=C_+$ the fundamental chamber, $b\in F^a$ and 
$v\in V_b\cap\overline{C_+}$. The jump conditions 
then hold for $p=1$ in view of the first line of \eqref{jsimple}.
For $p$ of degree one it suffices to check the jump conditions for $p=Db^\vee$
(see \eqref{jumpconditionsdegree1}),
which is a direct consequence of the second line of \eqref{jsimple}.

By Proposition \ref{elliptic} we conclude that 
$h\in F_{dr}^{\omega,k}(M)_\lambda$,
hence $h=T^{k,M}(h_{C_+})=T^{k,M}(g)=f$. Thus $f$ is $W^a$-invariant.
\end{proof}
\begin{rema}
For $M=\mathbb{I}$ the conditions \eqref{jsimple}
for $g\in E(\lambda)$ simplify to
$(\partial_{Da^\vee}g)(v)=-k_ag(v)$ and $\omega g=g$
for $a\in F^a$, $v\in V_a\cap\overline{C_+}$ and $\omega\in \Omega$.
\end{rema}

\subsection{The Bethe ansatz equations}

Proposition \ref{BAEdirect} gives a simple criterium to ensure
$W^a$-invariance of $f=T^{k,M}g\in F_{dr}^{\omega,k}(M)_\lambda$.
A different characterization is 
$Q_{k,M}(w)g=g$ for all $w\in W^a$. We reformulate
this now in explicit Bethe ansatz equations for the coefficients in the
plane wave expansion of $g\in E(\lambda)\otimes_{\mathbb{C}}M$.

We write $C_k^{reg}\subset V_{\mathbb{C}}^*$ for the $W$-invariant set
\[
C_k^{reg}=\{\lambda\in V_{\mathbb{C}}^* \,\, | \,\, 0\not=\lambda(Da^\vee)
\not=k_a\quad \forall\, a\in R^a \}.
\]
We assume throughout this subsection that $\lambda\in C_k^{reg}$
unless specified explicitly otherwise.
For such a spectral value $\lambda$ 
we then have, besides the plane wave basis $\{e^{w\lambda}\}_{w\in W}$
of $E(\lambda)$, the cocycle $\{J_w^k(\lambda)\}_{w\in W^a}$ in 
$\mathbb{C}[W^a]^\times$ to our disposal. We furthermore fix a left $W^a$-module
$M$ in the remainder of this subsection. 

The following lemma will be useful in the computations.
\begin{lem}\label{technical}
Let $\mu\in V_{\mathbb{C}}^*$ satisfying $\mu(\alpha^\vee)\not=0$
for all $\alpha\in R$.
For $a\in F^a$ and $m\in M$ we have
\[Q_{k,M}(s_a)\bigl(e^\mu\otimes m\bigr)=
e^{-a(0)\mu(Da^\vee)}e^{s_{Da}\mu}\otimes s_am+
k_a\left(\frac{e^{-a(0)\mu(Da^\vee)}e^{s_{Da}\mu}-e^\mu}
{\mu(Da^\vee)}\right)
\otimes m.
\]
\end{lem}
\begin{proof}
Since $Q_{k,M}(s_a)=s_a\otimes s_a-k_aI(a)\otimes\textup{Id}_M$ the
lemma follows from a direct computation using
\[I(a)(e^\mu)=\frac{e^\mu-e^{-a(0)\mu(Da^\vee)}e^{s_{Da}\mu}}{\mu(Da^\vee)}.
\]
\end{proof}

We now first consider the subspace 
$F_{ir}^k(M)_\lambda^W$ of
$Q_{k,M}(W)$-invariant elements in 
$F_{ir}^k(M)_\lambda$.
Recall the normalized intertwiners $J_w^k(\lambda)\in\mathbb{C}[W^a]$ for
$w\in W^a$ from subsection \ref{newform}.

Denote $1$ for the unit element of $W$.
\begin{prop}\label{coeffW}
Let $m_w\in M$ ($w\in W$). Then
\begin{equation}\label{psitilde}
\sum_{w\in W}e^{w\lambda}\otimes m_w\in
F_{ir}^k(M)_\lambda^W
\end{equation}
if and only if $m_w=J_w^k(\lambda)m_1$ $\forall\, w\in W$.
\end{prop}
\begin{proof}
We denote the left hand side of \eqref{psitilde} by $f_\lambda$.
Clearly $f_\lambda\in F_{ir}^k(M)_\lambda$ and $f_\lambda$ is $W$-invariant 
if and only if
\[Q_{k,M}(s_{\alpha})f_\lambda=f_\lambda\qquad \forall
\,\alpha\in F.
\] 
In the latter equations we expand both sides
sums of the plane waves $e^{w\lambda}$ ($w\in W$) with coefficients in $M$,
using Lemma \ref{technical} for the left hands side. Comparing coefficients
we get
that $f_\lambda\in F_{ir}^k(M)_\lambda^W$ if and only if 
\[((w\lambda)(\alpha^\vee)-k_\alpha)m_{s_{\alpha}w}=
((w\lambda)(\alpha^\vee)s_\alpha+k_{\alpha})m_w\qquad
\forall\, \alpha\in F,\,\, \forall\, w\in W.
\]
This can be rewritten as
$m_{s_{\alpha}w}=J_{s_\alpha}^k(w\lambda)m_w$ $\forall\,\alpha\in F$
$\forall\, w\in W$. By the cocycle property of $J_w^k(\lambda)$ this in turn is
equivalent to $m_w=J_w^k(\lambda)m_1$ $\forall\, w\in W$.
\end{proof}
\begin{cor}\label{coeffW2}
The assignment
\[m\mapsto \psi_\lambda^m:=\sum_{w\in W}e^{w\lambda}\otimes J_w^k(\lambda)m
\]
defines a complex linear isomorphism $M\rightarrow F_{ir}^k(M)_\lambda^W$.
Furthermore,
\begin{equation}\label{intermsofQ}
c_k(\lambda)\psi_\lambda^m=\sum_{w\in W}Q_{k,M}(w)(e^\lambda\otimes m)
\end{equation}
with $c_k(\lambda)$ the $c$-function
\[
c_k(\lambda)=\prod_{\alpha\in R^+}\frac{\lambda(\alpha^\vee)-k_\alpha}
{\lambda(\alpha^\vee)}.
\]
\end{cor}
\begin{proof}
The first statement is immediate from the previous proposition.
To prove the second statement we write $g_\lambda$ for the right hand side
of \eqref{intermsofQ}. Then $g_\lambda\in F_{ir}^k(M)_\lambda^W$, hence
\[g_\lambda=\psi_\lambda^{m^\prime}
\]
for a unique $m^\prime\in M$. By Lemma \ref{technical} one easily deduces that
\[J^k_{w_0}(\lambda)m^\prime=c_k(\lambda)J_{w_0}^k(\lambda)m
\]
for the longest Weyl group element $w_0\in W$ with respect to the 
basis $F$ of $R$. Hence $m^\prime=c_k(\lambda)m$, which concludes the proof.
\end{proof}
\begin{rema}
{\bf (i)}
The proofs of Proposition \ref{coeffW} and Corollary
\ref{coeffW2} are based on the methods from \cite[Prop. 1.4]{O}.\\ 
{\bf (ii)} 
For $M=\mathbb{I}$ we write $\psi_\lambda$ 
for $\psi_\lambda^m$ with $m=1$. Since $j_w^k(\lambda)=c_k(\lambda)^{-1}
c_k(w\lambda)$ for $w\in W$ (see \eqref{jfiniteexplicit})
we get
\[\psi_\lambda=\frac{1}{c_k(\lambda)}\sum_{w\in W}c_k(w\lambda)e^{w\lambda},
\]
in accordance with \cite{Ga,O,HO}.
\end{rema}

We next analyze when $\psi_\lambda^m\in F_{ir}^k(M)_\lambda^W$ is $W^a$-invariant.
We start with the following preliminary lemma.
\begin{lem}\label{snul}
For $m\in M$ we have 
\begin{equation}\label{1}
Q_{k,M}(s_{0})\psi_\lambda^m=\psi_\lambda^m
\end{equation}
if and only if
\begin{equation}\label{2}
J_{s_0w}^k(\lambda)m=e^{-(w\lambda)(\theta^\vee)}J_{s_{\theta}w}^k(\lambda)m\qquad 
\forall\,w\in W.
\end{equation}
\end{lem}
\begin{proof}
Substitute the plane wave expansion $\psi_\lambda^m=\sum_{w\in W}
e^{w\lambda}\otimes J_w^k(\lambda)m$
in the equality $Q_{k,M}(s_0)\psi_\lambda^m=\psi_\lambda^m$ and expand
again both sides as sum of plane waves, 
using Lemma \ref{technical} 
for the left hand side. Equating the coefficients
gives the equivalence between 
\eqref{1} and \eqref{2} (compare with 
the proof of Proposition \ref{coeffW}).
\end{proof}
\begin{thm}[Bethe ansatz equations]\label{BAEthm}
Let $\lambda\in C_k^{reg}$ and $m\in M$. We have
$\psi_\lambda^m\in F_{ir}^k(M)_\lambda^{W^a}$
if and only if
\begin{equation}\label{BAE}
J_{y}^k(\lambda)m=e^{\lambda(y)}m\qquad \forall\,y\in Y.
\end{equation}
\end{thm}
\begin{proof}
Since $s_0=s_{\theta}t_{-\theta^\vee}$, the cocycle condition for 
$J_w^k(\lambda)$ ($w\in W^a$) gives 
\[J_{s_0w}^k(\lambda)=J_{s_{\theta}w}^k(\lambda)J_{-w^{-1}(\theta^\vee)}^k(\lambda),
\qquad \forall\, w\in W.
\]
By Lemma \ref{snul} we conclude that $\psi_\lambda^m\in 
F_{ir}^k(M)_\lambda^{W\ltimes Q^\vee}$ if and only if
\begin{equation}\label{almostQ}
J_{y}^k(\lambda)m=e^{\lambda(y)}m
\end{equation}
for $y\in Q^\vee$ of the form $y=-w^{-1}(\theta^\vee)$ ($w\in W$).
The co-root lattice $Q^\vee$ is generated by $W\theta^\vee$ and $y\mapsto
J_y^k(\lambda)$ defines 
a group homomorphism $Y\rightarrow \mathbb{C}[W^a]^\times$, 
hence the theorem is correct when $Y=Q^\vee$. 

Fix $\omega\in\Omega$ and write $\omega=t_y\sigma$ with
$y\in Y$ and $\sigma\in W$. For $w\in W$ we then have
\begin{equation*}
\begin{split}
\omega(e^{w\lambda})&=e^{-(\sigma w\lambda)(y)}e^{\sigma w\lambda},\\
\omega J_w^k(\lambda)m&=J_{\omega w}^k(\lambda)m=J_{t_y\sigma w}^k(\lambda)m,
\end{split}
\end{equation*}
where the first identity is in $C^\omega(V)$ (with the
standard $W^a$-action). 
Using the plane wave expansion $\psi_\lambda^m=\sum_{w\in W}
e^{w\lambda}\otimes J_w^k(\lambda)m$ we conclude that
\[Q_{k,M}(\omega)\psi_\lambda^m=\psi_\lambda^m\]
if and only if
\[
J_{t_yw}^k(\lambda)m=e^{(w\lambda)(y)}J_w^k(\lambda)m\qquad \forall\, w\in W.
\]
By the cocycle condition for $J_w^k(\lambda)$ ($w\in W^a$), 
this in turn is equivalent to
\[J_{wy}^k(\lambda)m=e^{\lambda(wy)}m\qquad \forall\, w\in W.
\]
If $\mathcal{S}$ denotes the set of $y\in Y$ for which there
exists an $\omega\in\Omega$ of the form $\omega=t_y\sigma$ 
($\sigma\in W$), then we conclude that $\psi_\lambda^m\in 
F_{ir}^k(M)_\lambda^{W^a}$
if and only if \eqref{almostQ} holds for all $y$ in the sublattice $L$
of $Y$ generated by $Q^\vee$ and the $W$-orbit of $\mathcal{S}$.
Since $L=Y$, this concludes the proof of the theorem.
\end{proof}
For nontrivial $W^a$-modules $M$ various special instances
of the Bethe ansatz equations \eqref{BAE} can be found in
the literature, see, e.g., \cite{McG,Y,S,CY}. For $M=\mathbb{I}$
the theorem becomes the following statement.
\begin{cor}\label{trivcase} 
\[\psi_\lambda=\frac{1}{c_k(\lambda)}\sum_{w\in W}c_k(w\lambda)e^{w\lambda}\in 
E(\lambda)
\]
is $Q_{k,\mathbb{I}}(W^a)$-invariant if and only if the spectral parameter
$\lambda\in C_k^{reg}$ satisfies the Bethe ansatz equations
\[j_{y}^k(\lambda)=e^{\lambda(y)}\qquad \forall\, y\in Y
\]
with $j_{y}^k(\lambda)$ given by \eqref{jexplicit}.
\end{cor}
Corollary \ref{trivcase} generalizes the Bethe ansatz equations
from \cite{EOS}, which dealt with the special case $k_0=k_\theta$ and
$Y=Q^\vee$. For root systems $R$ of classical type it goes back to
\cite{LL,Ga}. 

\subsection{Solutions to the Bethe ansatz equations}

For $M=\mathbb{I}$, $k_0=k_\theta$ and
$Y=Q^\vee$, the set of spectral values $\lambda$ solving the 
Bethe ansatz equations \eqref{BAE}
has been analyzed in detail. For $R$ the root system of type $A$
it goes back to \cite{YY}. The solutions are parametrized by
the extrema of a family of strictly convex functions. Such phenomena
happen in various other quantum integrable models, such as the
Gaudin model \cite{RV}.

We now analyze the Bethe ansatz equations \eqref{BAE}
for $M=\textup{Fun}_{\mathbb{C}}(W^a)$ the complex valued functions on $W^a$,
viewed as left $W^a$-module by the left regular action
$\bigl(L(w)f\bigr)(w^\prime)=f(w^{-1}w^\prime)$ for $w,w^\prime\in W^a$ and 
$f\in\textup{Fun}_{\mathbb{C}}(W^a)$. 
We denote $\textup{Fun}_{\mathbb{C}}^R(W^a)$ for the vector space
$\textup{Fun}_{\mathbb{C}}(W^a)$ endowed with the left $W^a$-action
$\bigl(R(w)f\bigr)(w^\prime)=f(w^\prime w)$.

Consider for $\lambda\in V_{\mathbb{C}}^*$
the $A(k)$-module $F_{ir}^k(\mathbb{I})_\lambda$ 
as $W^a$-module by restricting the representation
map $Q_{k,\mathbb{I}}$ to $W^a$. We view the complex linear dual 
$F_{ir}^k(\mathbb{I})_\lambda^*$ as $W^a$-module with respect to the 
associated contragredient action.
We have a complex linear map 
$\Xi_{\lambda}: F_{ir}^k(\mathbb{I})_\lambda^*\otimes
F_{ir}^k(\mathbb{I})_\lambda
\rightarrow \textup{Fun}_{\mathbb{C}}(W^a)$ defined by
\[\Xi_\lambda(g\otimes v):=g\bigl(Q_{k,\mathbb{I}}(\cdot)v\bigr).
\]
It intertwines the
$W^a\times W^a$-action on $F_{ir}^k(\mathbb{I})_\lambda^*\otimes 
F_{ir}^k(\mathbb{I})_\lambda$
with the $W^a\times W^a$-action $L\times R$ on $\textup{Fun}_{\mathbb{C}}(W^a)$.
In particular, $\Xi_\lambda$ is a $W^a$-module morphism
with respect to the conjugation action
$(wf)(w^\prime)=f(w^{-1}w^{\prime}w)$ on $\textup{Fun}_{\mathbb{C}}(W^a)$.

Denote $\epsilon_\lambda: \mathbb{I}\rightarrow 
F_{ir}^k(\mathbb{I})_\lambda\otimes 
F_{ir}^k(\mathbb{I})_\lambda^*$ for the co-evaluation map. It is the $W^a$-module
morphism defined by 
$\epsilon_\lambda(1)=\sum_iv_i\otimes v_i^*$, where $\{v_i\}_i$ is
a basis of $F_{ir}^k(\mathbb{I})_\lambda$
and $\{v_i^*\}_i$ is the corresponding dual basis. 

\begin{thm}\label{BAEW}
Let $\lambda\in V_{\mathbb{C}}^*$. The mapping 
\[g\mapsto 
(\textup{Id}_{F_{ir}^k(\mathbb{I})_\lambda}\otimes \Xi_\lambda)
(\epsilon_\lambda\otimes\textup{Id}_{F_{ir}^k(\mathbb{I})_\lambda})g
\]
defines an isomorphism
\[\Psi_\lambda: F_{ir}^k(\mathbb{I})_\lambda\rightarrow 
F_{ir}^k(\textup{Fun}_{\mathbb{C}}(W^a))_\lambda^{W^a}
\]
of left $W^a$-modules, where 
$F_{ir}^k(\textup{Fun}_{\mathbb{C}}(W^a))_\lambda^{W^a}$
is regarded as a $W^a$-submodule of 
$E(\lambda)\otimes_{\mathbb{C}}\textup{Fun}_{\mathbb{C}}^R(W^a)$
with respect to the action $w\mapsto \textup{Id}\otimes R(w)$.
\end{thm}
\begin{proof}
By construction $\Psi_\lambda$ defines a complex linear map
\[\Psi_\lambda: F_{ir}^k(\mathbb{I})_\lambda\rightarrow
\bigl(F_{ir}^k(\mathbb{I})_\lambda\otimes 
\textup{Fun}_{\mathbb{C}}(W^a)\bigr)^{W^a}.
\]
By Proposition \ref{invariantsequal} we have
\begin{equation}\label{identification}
\bigl(F_{ir}^k(\mathbb{I})_\lambda\otimes 
\textup{Fun}_{\mathbb{C}}(W^a)\bigr)^{W^a}=
F_{ir}^k(\textup{Fun}_{\mathbb{C}}(W^a))_\lambda^{W^a}.
\end{equation}
The resulting linear map
\[\Psi_\lambda: F_{ir}^k(\mathbb{I})_\lambda\rightarrow 
F_{ir}^k(\textup{Fun}_{\mathbb{C}}(W^a))_\lambda^{W^a}
\]
is easily seen to be a morphism of left $W^a$-modules. 
It remains to show that $\Psi_\lambda$
is an isomorphism. Define the linear mapping $\Phi_\lambda: 
F_{ir}^k(\textup{Fun}_{\mathbb{C}}(W^a))_\lambda^{W^a}
\rightarrow F_{ir}^k(\mathbb{I})_\lambda$
by 
\[\Phi_\lambda(\sum_jg_j\otimes h_j)=\sum_jh_j(1)g_j,
\]
where $1$ is the unit element of $W^a$.
Clearly $\Phi_\lambda$ is a left
inverse of $\Psi_\lambda$. By the alternative
characterization \eqref{identification} we have for
$\sum_jg_j\otimes h_j\in F_{ir}^k(\textup{Fun}_{\mathbb{C}}(W^a))_\lambda^{W^a}$,
\[
\sum_jh_j(w)g_j=\sum_jh_j(1)Q_{k,\mathbb{I}}(w)g_j,\qquad \forall\, w\in W^a,
\]
which implies that $\Phi_\lambda$ is also the right inverse of
$\Psi_\lambda$.
\end{proof}
Let $\lambda\in C_k^{reg}$.
Theorem \ref{BAEW} gives rise to the following explicit description of the
solutions of the Bethe ansatz equations \eqref{BAE}
for $M=\textup{Fun}_{\mathbb{C}}(W^a)$.

For a common eigenfunction of the scalar
Dunkl type operators, $f=T^{k,\mathbb{I}}(g)\in 
F_{dr}^{\omega,k}(\mathbb{I})_\lambda
=T^{k,\mathbb{I}}(F_{ir}^k(\mathbb{I})_\lambda)$, we have
\[f(u^{-1}\,\cdot)|_{C_+}=\sum_{w\in W}c_{g,w}(u)e^{w\lambda}|_{C_+},
\qquad u\in W^a,
\]
with $c_{g,w}\in\textup{Fun}_{\mathbb{C}}(W^a)$ determined by
\[Q_{k,\mathbb{I}}(u)g=\sum_{w\in W}c_{g,w}(u)e^{w\lambda},\qquad
\forall\, u\in W^a.
\]
\begin{cor}\label{BAEWcor}
For $\lambda\in C_k^{reg}$ we have
\[\Psi_\lambda(g)=\sum_{w\in W}e^{w\lambda}\otimes c_{g,w},\qquad
\forall\, g\in F_{ir}^k(\mathbb{I})_\lambda.
\]
In particular, $c_{g,w}=L(J_w^k(\lambda))c_{g,1}$ for $w\in W$
and $g\in F_{ir}^k(\mathbb{I})_\lambda$, and 
the functions $c_{g,1}\in\textup{Fun}_{\mathbb{C}}(W^a)$
($g\in F_{ir}^k(\mathbb{I})_\lambda$) give all the solutions of the
Bethe ansatz equations \eqref{BAE} for the $W^a$-module 
$M=\textup{Fun}_{\mathbb{C}}(W^a)$.
\end{cor}
\begin{rema}
Corollary \ref{BAEWcor} should be compared to 
the observation of Yang \cite{Y} (see also \cite{S,LLP,LLT})
that an arbitrary eigenfunction of the quantum Hamiltonian 
of the one dimensional quantum particle system with 
delta function potentials gives rise to 
a symmetric eigenfunction of the quantum Hamiltonian
of an associated quantum spin-particle model.
\end{rema}

The principal series modules of $W^a$ are defined as follows.
For $t\in V_{\mathbb{C}}^*$ the principal series module of $W^a$
with central character $t$ is the vector space
$M(t)=\bigoplus_{w\in W}\mathbb{C}v_w(t)$ with action given by
\begin{equation*}
\begin{split}
uv_w(t)&=v_{uw}(t),\,\,\,\,\qquad\qquad u\in W,\\
yv_w(t)&=e^{t(w^{-1}y)}v_w(t),\qquad y\in Y
\end{split}
\end{equation*}
for $w\in W$. Observe that $M(t)$ is unitarizable iff 
$t\in\sqrt{-1}V^*$, in which case the corresponding scalar product
is given by $\bigl(v_u(t), v_w(t)\bigr):=\delta_{u,w}$ for $u,w\in W$
(conjugate linear in the second component).
  
Yang \cite{Y} derived and studied the Bethe ansatz equations
\eqref{BAE} for root system $R$ of type $A$ and for $M$ the
principal series module $M(0)$. More generally, a 
natural problem is to describe
the solutions of the Bethe ansatz equations \eqref{BAE} for the 
principal series module $M(t)$. 
We make a modest start here with the following observation.

\begin{prop}
Fix $\lambda\in C_k^{reg}$ and let $M$ be a unitary $W^a$-module.
Assume that $k_a\in\mathbb{R}_{<0}$ for all $a\in R^a$
(which corresponds to repulsive delta function interactions in the physical
interpretation).
If there exists an $0\not=m\in M$ satisfying the Bethe ansatz
equations \eqref{BAE}, then $\lambda\in\sqrt{-1}V^*$.
\end{prop}
\begin{proof}
We assume that 
$\lambda=\mu+\sqrt{-1}\nu\in C_k^{reg}$ with $\mu,\nu\in V^*$ and 
$\mu(\alpha^\vee)\geq 0$ for all $\alpha\in R^+$.
It suffices to prove the theorem under these additional 
assumptions; indeed, if $0\not=m\in M$ satisfies \eqref{BAE} 
then $0\not=n:=J_w^k(\lambda)m\in M$ ($w\in W$) satisfies 
\[
J_y^k(w\lambda)n=e^{(w\lambda)(y)}n,\qquad \forall\, y\in Y.
\]

Fix $k<0$ and $u\in\mathbb{C}$ with $\textup{Re}(u)\geq 0$.
Denote  $\bigl(\cdot,\cdot\bigr)$ for the scalar product on $M$, and
$\|\cdot\|$ for the corresponding norm.
For a simple affine root $a\in F^a$ and $0\not=n\in M$ we then have
\[
\left\|\left(\frac{us_a+k}{u-k}\right)n\right\|=\left|\frac{u+k}{u-k}\right|
\|n\|,
\]
which is $\leq \|n\|$ and $=\|n\|$ iff $\textup{Re}(u)=0$.

We are first going to apply this estimate to
show that $J_{\theta^\vee}^k(\lambda)$
has operator norm $\leq 1$.
Write $s_\theta=s_{\alpha_1}s_{\alpha_2}\cdots s_{\alpha_r}$ 
($\alpha_j\in F$) for a reduced expression of $s_\theta\in W$. Then
\[
R^+\cap s_\theta R^-=\{\beta_1,\ldots,\beta_r\},
\qquad \beta_j=s_{\alpha_r}s_{\alpha_{r-1}}\cdots s_{\alpha_{j+1}}(\alpha_j).
\]
Since $t_{\theta^\vee}=s_0s_\theta$ the cocycle condition gives
\[J_{\theta^\vee}^k(\lambda)=\left(\frac{\lambda(\theta^\vee)s_0+k_0}
{\lambda(\theta^\vee)-k_0}\right)\left(\frac{\lambda(\beta_1)s_{\alpha_1}+
k_{\alpha_1}}{\lambda(\beta_1)-k_{\alpha_1}}\right)
\cdots
\left(\frac{\lambda(\beta_r)s_{\alpha_r}+k_{\alpha_r}}{\lambda(\beta_r)-
k_{\alpha_r}}\right).
\]
The previous paragraph thus implies that 
$\|J_{\theta^\vee}^k(\lambda)\|\leq 1$.

Let $0\not=m\in M$ be a solution
of the Bethe ansatz equations \eqref{BAE}.
Then
\[|e^{\lambda(\theta^\vee)}|=\frac{\|J_{\theta^\vee}^k(\lambda)m\|}{\|m\|}
\leq 1.
\]
Since $\lambda=\mu+\sqrt{-1}\nu$ with $\mu(\alpha^\vee)\geq 0$ for all
$\alpha\in R^+$ we obtain $\mu(\theta^\vee)=0$. We have
$\theta^\vee=\sum_{\alpha\in F}n_\alpha\alpha^\vee$ with $n_\alpha\geq 1$,
hence $\mu(\alpha^\vee)=0$ for all $\alpha\in F$.
In particular, $\mu\in V^*$ is $W$-invariant and
\[
\|J_w(\lambda)n\|=\|J_w(\sqrt{-1}\nu)n\|=\|n\|,\qquad \forall\, w\in W, 
\forall\, n\in M.
\]
Let now $\omega\in\Omega$ and write $\omega=\sigma t_y$ with 
$\sigma\in W$ and $y\in Y$. Then 
\begin{equation*}
\begin{split}
\|m\|=\|\omega m\|&=\|J_\omega^k(\lambda)m\|\\
&=\|J_\sigma^k(\lambda)J_y^k(\lambda)m\|\\
&=|e^{\lambda(y)}|\|m\|,
\end{split}
\end{equation*}
hence $\mu(y)=0$. As in the proof of Theorem 
\ref{BAEthm} we conclude that $\mu(Y)=0$, hence $\mu=0$.
\end{proof}

\noindent
{\bf Acknowledgments:} Stokman is supported by the Netherlands Organization for Scientific Research (NWO) in the VIDI-project "Symmetry and modularity in exactly solvable models". He thanks Pavel Etingof for useful discussions.

\end{document}